\documentclass[11pt,reqno]{article}
\usepackage{amssymb, amsmath, amsthm, amsfonts, amscd, epsfig}
\usepackage{color}

\usepackage[english]{babel}
\usepackage{bm}
\usepackage{bbm}
\usepackage{graphicx, epsfig}

\usepackage{subcaption}
\usepackage[export]{adjustbox}
\usepackage[titletoc,toc,title]{appendix}
\usepackage[normalem]{ulem} 

\newtheorem{theorem}{Theorem}[section]
\newtheorem{cor}[theorem]{Corollary}
\newtheorem{lemma}[theorem]{Lemma}                                                                                                                                                                                                                                                                             

\newtheorem{definition}{Definition}

\newtheorem{remark}{Remark}

\def\ep{\epsilon}

\newcommand{\Bp}{\mathbf{p}}

\newcommand{\Bc}{\mathbf{c}}

\newcommand{\Be}{\mathbf{e}}

\newcommand{\Om}{\Omega}
\newcommand{\Bs}{\mathbf{s}}
\newcommand{\Bx}{\mathbf{x}}
\newcommand{\By}{\mathbf{y}}

\newcommand{\RR}{\mathbb{R}}
\newcommand{\CC}{\mathbb{C}}
\newcommand{\NN}{\mathbb{N}}

\newcommand{\Scal}{\mathcal{S}}

\newcommand{\Kcal}{\mathcal{K}}

\newcommand{\p}{\partial}

\newcommand{\pd}[2]{\frac {\p #1}{\p #2}}
\newcommand{\ds}{\displaystyle}
\newcommand{\eqnref}[1]{(\ref {#1})}

\renewcommand{\qed}{\hfill $\Box$ \medskip}

\newcommand{\beq}{\begin{equation}}
\newcommand{\eeq}{\end{equation}}
\newcommand{\RN}[1]{%
  \textup{\uppercase\expandafter{\romannumeral#1}}%
}

\numberwithin{equation}{section}
\numberwithin{figure}{section}

\begin{document}

\newcommand{\TheTitle}{Electric field concentration in the presence of an inclusion with eccentric core-shell geometry}
\newcommand{\TheAuthors}{J. Kim and M. Lim}

\title{{\TheTitle}\thanks{{This work is supported by the Korean Ministry of Science, ICT and Future Planning through NRF grant No. 2016R1A2B4014530 (to M.L. and J.K.).}}}
\author{
Junbeom Kim\thanks{\footnotesize Department of Mathematical Sciences, Korea Advanced Institute of Science and Technology, Daejeon 34141, Korea ({kjb2474@kaist.ac.kr}, {mklim@kaist.ac.kr}).}\footnotemark[2] \and Mikyoung Lim\footnotemark[2]}

\maketitle

\begin{abstract}
In this paper we analyze the gradient blow-up of the solution to the conductivity problem in two dimensions in the presence of an inclusion with eccentric core-shell geometry. Assuming that the core and shell have circular boundaries that are nearly touching, we derive an asymptotic formula for the solution in terms of the single and double layer potentials with image line charges. We also deduce an integral formula with image line charges for the problem relating to two nearly touching separated conductors. 
\end{abstract}

\noindent {\footnotesize {\bf Mathematics subject classification
(MSC2000): 35J25, 73C40} }


\tableofcontents

\section{Introduction}
We consider the conductivity problem in the presence of an inclusion with eccentric core-shell geometry which consists of a core and a surrounding shell, where the core and shell have circular boundaries that are nearly touching. Specifically, we let $B_i$ and $B_e$ be two disks such that $B_e$ contains $B_i$, as drawn in Figure \ref{fig:geometry}, where the core $B_i$ and shell $B_e\setminus B_i$ have the constant conductivities $k_i$ and $k_e$ satisfying $0<k_i,k_e\neq 1<\infty$. The aim of the paper is to derive an asymptotic formula for the solution to the conductivity problem
\begin{equation} \label{Eqn}
\left\{ \begin{array}{ll}
\ds \nabla \cdot \sigma \nabla u = 0 \qquad& \textrm{in } \mathbb{R}^2, \\
\ds u(\bold{x}) - H(\bold{x}) = O(|\bold{x}|^{-1}) \qquad& \textrm{as } |\bold{x}| \to \infty,
\end{array} \right.
\end{equation}
as the distance between $\p B_i$ and $\p B_e$ degenerates to zero.
Here, $\sigma$ denotes the conductivity profile $\sigma = k_i \chi(B_i) + k_e \chi(B_e \setminus B_i) + \chi(\mathbb{R}^2 \setminus B_e)$, $H$ is an entire harmonic function, and $\chi(D)$ represents the characteristic function of a set $D$.

The electrostatic interaction problem has been studied for a long time, including \cite{Max}.
For two separated spheres, the series expression by separation of variables in bispherical coordinates and the method of image charges were developed \cite{Dav,Jef,Smy}. The electric field may blow up as the distance tends toward zero \cite{keller}, and it causes difficulty in the numerical computation. There have been many attempts to overcome this difficulty and to compute the asymptotic behavior of the electric field \cite{Bat,keller,McP2,McP,Pol,Pol2}.
For two dimensions, this problem can be also considered in relation to the stress concentration in composite materials. 
In composite structures, which consist of a matrix and inclusions, it is common for inclusions to be closely located. 
Having extreme material parameters is a necessary condition for inclusions to induce the electric field or the stress tensor blow-up. Li and Vogelius showed that the electric field is bounded if the inclusions' conductivities are finite and strictly positive \cite{LV}, and Li and Nirenberg extended this result to elliptic systems \cite{LN}.

In recent years, the gradient blow-up feature of the electric field in the small area between the inclusions' boundaries has attracted much attention due to its application to various imaging modalities, for example \cite{lassiter2008close,Rom}. In particular, the electric field concentration in the presence of two nearly touching separated inclusions has been intensively studied. We denote by $\ep$ the distance between the inclusions. The generic blow-up rate for conductors is $\ep^{-1/2}$ in two dimensions \cite{AKL,bab,BC,BT,keller,Y,Y2}, and $|\ep\ln\ep|^{-1}$ in three dimensions \cite{BLY,BLY2,LY}. 
In two dimensions, Ammari et al. deduced an optimal bound for two disks with arbitrary constant conductivities \cite{AKLLL,AKL}, and Yun obtained the blow-up rate for perfect conductors of general shapes \cite{Y,Y2}.
The asymptotic behavior of the gradient blow-up solution was studied for perfect conductors of circular shapes by Kang et al. \cite{KLY}, and of {arbitrary shapes by Ammari et al. \cite{ACKLY} and Kang et al. \cite{KLeeY}. For disks with arbitrary constant conductivities, Lim and Yu obtained an asymptotic formula for the electric field \cite{LYu2D}. The blow-up feature for perfect conductors of a bow-tie structure was investigated by Kang and Yun in \cite{kang2017optimal}.
	In three dimensions, the blow-up rate for perfect conductors of general shapes was obtained by Bao et al. \cite{BLY,BLY2}. For perfect conductors of spherical shapes, Ammari et al. showed that the potential difference in the conductors is necessary for the gradient blow-up \cite{ammari2007estimates}, an optimal bound for the electric field concentration was obtained by Lim and Yun \cite{LY}, and the asymptotic feature of the blow-up solution was investigated by Kang et al. \cite{KLY2} and by Lim and Yu \cite{LYu3D}. Lekner calculated the field enhancement rate by using the bispherical coordinates \cite{lekner11}. Recently, Ammari and Yu developed a numerical method to compute the blow-up solution \cite{AmmariYu}.
	The blow-up rate for insulating inclusions is the same as perfectly conducting ones in two dimensions, but is different in three dimensions. Yun showed that the rate in three dimensions is $1/\ep^{\frac{2-\sqrt{2}}{2}}$ \cite{Y3}.
	For the elasticity problem, Bao et al. obtained an upper bound in two dimensions \cite{bao2015gradient}, and this result was generalized to higher dimensions \cite{bao2017gradient}. Recently, the gradient blow-up term of the solution in two dimensions was verified for {convex hard inclusions} by Kang and Yu \cite{kang2017quantitative} and for circular holes by Lim and Yu \cite{lim2017stress}.

	An inclusion with core-shell geometry can also induce the gradient blow-up for the solution to the conductivity problem as the distance between the boundaries of the core and shell degenerates to zero. 
	In two dimensions, Ammari et al. obtained the blow-up rate $\ep^{-1/2}$ for core-shell geometry with circular boundaries \cite{AKLLL}.
	Recently, Li et al. obtained an optimal bound for inclusions of general shapes in two dimensions and higher \cite{li2017optimal}, and Bao et al. extended the result to the linear elasticity problem \cite{bao2017optimal}.
	It is worth mentioning that a plasmonic structure of core-shell geometry induces the shielding effect on a region away from the structure \cite{YL17}.

	Our goal in this paper is to characterize the electric field concentration for core-shell geometry with circular boundaries. We first show that only the linear term in $H(\Bx)$ contributes to the gradient blow-up. We then show that the gradient blow-up happens only in the shell area $B_e\setminus B_i$. We finally express $u$ by the single and double layer potentials with image line charges as follows: for $\Bx\in B_e\setminus B_i$,
	\begin{align*}
	&	u(\Bx)-H(\Bx)\\
	&=  \frac{C_x r_*^2}{2} \left[  \int_{((-\infty,0), \bold{p_1}]} \ln|\Bx-\Bs|\varphi_1(\Bs) \,d\sigma(\Bs) -\int_{[\bold{p}_2,\bold{c_i}]} \ln|\Bx-\Bs| \varphi_2(\Bs)\, d\sigma(\Bs)\right] \\ 
	& + \frac{C_y r_*^2}{2} \left[ \int_{((-\infty,0), \bold{p_1}]} \frac{ \langle \Bx-\Bs, \Be_2\rangle}{|\Bx-\Bs|^2} \psi_1(\Bs)\, d\sigma(\Bs)  -\int_{[\bold{p}_2,\bold{c_i}]} \frac{ \langle \Bx-\Bs, \Be_2 \rangle}{|\Bx-\Bs|^2} \psi_2(\Bs)\, d\sigma(\Bs)\right]  + r(\Bx),
	\end{align*}
	where $C_x,\, C_y,\, r_*^2$ are some constants depending on the background potential function $H$ and the geometry of the inclusion, and $|\nabla r|$ is uniformly bounded independently of $\ep,\, \Bx$. The points $\Bp_1,\Bp_2$ and $\Bc_i$ are defined according to $B_i$ and $B_e$; see Theorem \ref{asymptotic} for further details. 
	In the derivation, we use the bipolar coordinate system due to its convenience to solve the conductivity problem in the presence of two circular interfaces. 
	We also deduce an integral expression with image line charges for the problem relating to two separated disks.

	The paper is organized as follows. In section \ref{sec:bipolar}, we explain the core-shell geometry and the associated bipolar coordinate system. 
	We decompose the solution $u$ into a singular term and a regular term in section \ref{sec:series_bipolar} and derive an asymptotic formula for the solution in terms of bipolar coordinates in section \ref{sec:bipolarapprox}. 
	Section \ref{sec:cartesian} is about the image line charge formula, and we conclude with some discussion.

	\section{The bipolar coordinate system} \label{sec:bipolar}

	\subsection{Geometry of the core-shell structure}
	We let $B_i$ and $B_e$ be the disks centered at $\Bc_i$ and $\Bc_e$ with the radii $r_i$ and $r_e$, respectively. As stated before, we assume that $B_e$ contains $B_i$ and that the distance between the boundary circles, say $\ep=\mbox{dist}(\p B_i, \p B_e)$, is small. After appropriate shifting and rotation, the disk centers are $\Bc_i=(c_i,0)$ and $\Bc_e=(c_e,0)$ with
	\beq \label{disks:center}
	\ds c_i = \frac{r_e ^2 - r_i ^2 - (r_e - r_i - \epsilon)^2}{2(r_e-r_i-\epsilon)}, \quad
	\ds c_e = c_i +r_e- r_i - \epsilon .
	\eeq
	Note that $r_e>r_i>0$ and $c_e>c_i>0$ as shown in Figure \ref{fig:geometry}.
	The closest points on $\p  B_i$ and $\p B_e$ between $\p B_i$ and $\p B_e$ are, respectively, $\Bx_i=(x_i,0)$ and $\Bx_e=(x_e,0)$, with
	\begin{align*}
	x_i =c_i-r_i=\frac{2\ep r_e-\ep^2}{2(r_e-r_i-\ep)},\quad
	x_e=c_e-r_e=\frac{2\ep r_i+\ep^2}{2(r_e-r_i-\ep)}.
	\end{align*}
	Note that $x_i,x_e=O(\ep)$. 
	\begin{figure}[!h]
		\begin{center}
			\includegraphics[width=6.5cm]{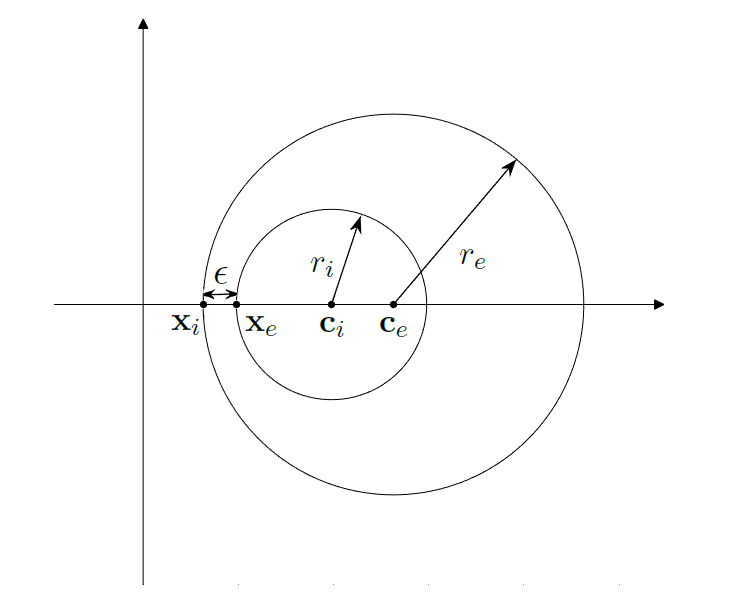}\hskip 1cm
			\includegraphics[width=6.5cm]{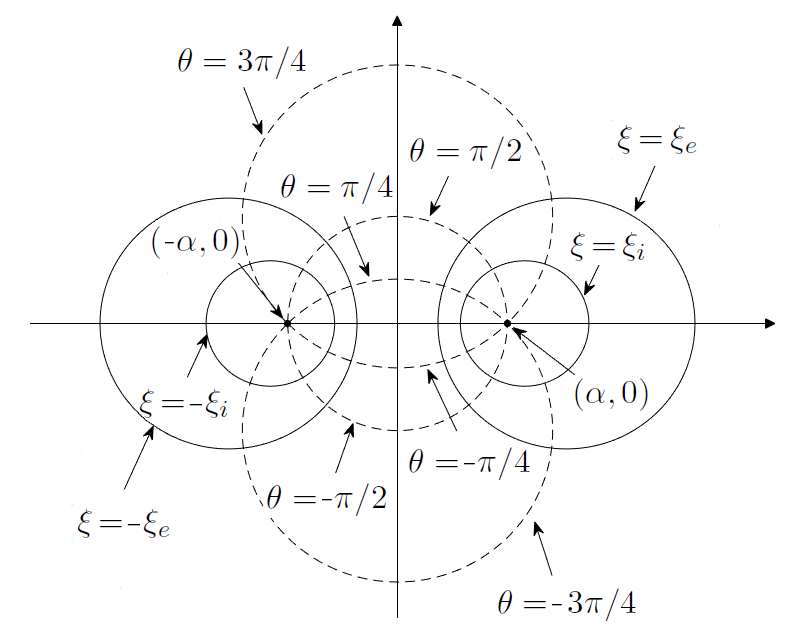}
			\caption{An inclusion with eccentric core-shell geometry (left) and the coordinate level curves of associated bipolar coordinates (right).}\label{fig:geometry}
		\end{center}
	\end{figure}

	For a disk $D$ centered at $\bold{c}$ with radius $r$, the reflection of $\Bx=(x,y)\in\RR^2$ across $\p D$ is
	$$
	R_{\p D} (\bold{x}) = \frac{r^2 (\bold{x}-\bold{c})}{|\bold{x}-\bold{c}|^2}+\bold{c}.
	$$
	We also define $R_{\p D}(z)\in\CC$ for $z=x+\rm{i}y$ such that $R_{\p D}(\Bx)=(\Re\{R_{\p D} (z)\},\Im\{R_{\p D} (z)\})$. Here the symbols $\Re\{w\}$ and $\Im\{w\}$ indicate the real and imaginary parts of a complex number $w$, respectively.
	We then define the reflection of a function $f$ across $\p D$ as 
	$$R_{\p D}[f] (\bold{x}) = f\left( R_{\p D}(\bold{x})\right)\qquad\mbox{for }\Bx\in\RR^2,$$ and the combined reflections as
	$$R_{\p D_1}R_{\p D_2}\cdots R_{\p D_n}[f](\Bx)=f\left((R_{\p D_n}\circ R_{\p D_{n-1}}\cdots \circ R_{\p D_1})(\Bx)\right).$$

	Let $R_i$ and $R_e$ denote the reflections across the circles $\p B_i$ and $\p B_e$, respectively. Then, one can easily show that the combined reflections $R_i \circ R_e$ and $R_e \circ R_i$  have the same fixed points $\Bp_1 = (-\alpha,0)$ and $\Bp_2=(\alpha,0)$ with
	\beq \label{defalp}
	\alpha = \frac{\sqrt{\epsilon(2r_i+\epsilon)(2r_e-\epsilon)(2r_e-2r_i-\epsilon)}}{2(r_e-r_i-\epsilon)}.
	\eeq
	One can also easily show that
	\beq
	\alpha= r_* \sqrt{ \epsilon} + O(\epsilon \sqrt{\epsilon}),\quad \mbox{where }
	r_* = \sqrt{\frac{2r_i r_e}{r_e-r_i}}.\label{def:rstar}
	\eeq

	The boundaries of $B_i$ and $B_e$ are then the coordinate level curves $\{\xi= \xi_i\}$ and $\{\xi = \xi_e\}$, respectively, in the bipolar coordinate system $(\xi,\theta)$ that will be defined below,
	with
	\beq\label{xi_ie}
	\xi_i = \ln |c_i + \alpha| - \ln (r_i),\quad \xi_e = \ln |c_e + \alpha| - \ln (r_e).
	\eeq
	%
	For later use we set
	\begin{equation}\label{def:xik}
	\begin{cases}
	\ds\xi_{i,k}=2k(\xi_i-\xi_e)+\xi_i,\\
	\ds\xi_{e,k} = 2k(\xi_i-\xi_e)+\xi_e
	\end{cases}
	\end{equation} 
	for each $k\in\NN$. 
	One can show $$\xi_i>\xi_e>0,$$ so that $\xi_{i,k}\geq\xi_i$ and $\xi_{e,k+1}\geq \xi_e$.

	\subsection{The bipolar coordinate system associated with the eccentric core-shell structure}
	We set the bipolar coordinate system $(\xi, \theta) \in \mathbb{R} \times (-\pi, \pi]$ with respect to the two poles located at $\Bp_1=(-\alpha,0)$ and $\Bp_2=(\alpha,0)$: for any point $\Bx=(x,y)$ in the Cartesian coordinate system, the coordinates $(\xi,\theta)$ are defined such that
	\beq\label{def:bipolar}e^{\xi+\rm{i}\theta}=\frac{\alpha+z}{\alpha-z}\quad\mbox{with } z=x+\rm{i}y.\eeq
	In other words
	\begin{align}\label{bipolar:z}
	z&=\alpha \frac{e^{\xi+{\rm i}\theta}-1}{e^{\xi+{\rm i}\theta}+1}= \alpha \frac{\sinh \xi}{\cosh \xi + \cos \theta}+{\rm i} \alpha \frac{\sin \theta}{\cosh \xi + \cos \theta}.
	\end{align}
	We write either $\Bx=\Bx(\xi,\theta)$ or $z=z(\xi,\theta)$ to specify bipolar coordinates if necessary. For a function $u=u(\Bx)$, we may write $u(\xi,\theta)$ to indicate $u(\Bx(\xi,\theta))$ for notational simplicity. Figure \ref{fig:geometry} illustrates the coordinate level curves of the bipolar coordinate system. Note that the two level curves $\{\theta=\frac{\pi}{2}\}$ and $\{\theta=-\frac{\pi}{2}\}$ form one circle that is centered at the origin with radius $\alpha$. Therefore, we have that
	\beq\label{thetalevel}
	\begin{cases}
		\ds |z(\xi,\theta)|\leq \alpha\qquad\mbox{if }|\theta|\leq \frac{\pi}{2},\\[2mm]
		\ds |z(\xi,\theta)|>\alpha\qquad\mbox{if }|\theta|>\frac{\pi}{2}.
	\end{cases}
	\eeq

	One can easily see that the two scale factors coincide, {\it i.e.},
	$$\left|\pd{\Bx}{\xi}\right|=\left|\pd{\Bx}{\theta}\right|=\frac{1}{h(\xi,\theta)},$$
	where $$h(\xi, \theta) = \frac{\cosh \xi + \cos \theta }{\alpha}.$$
	In particular, the components of $\p \Bx/\p \xi$
	\beq\label{eqn:xdxi}
	\frac{\p x}{\p \xi} = \alpha \frac{1+\cos \theta \cosh\xi}{(\cosh\xi + \cos \theta)^2},\quad\frac{\p y}{\p \xi} = -\alpha \frac{\sin \theta \sinh\xi}{(\cosh\xi + \cos \theta)^2}
	\eeq
	satisfy that
	$\big| \frac{\p x}{\p \xi} \big|, \big| \frac{\p y}{\p \xi} \big|\leq  \frac{1}{h(\xi,\theta)}.$
	The gradient for a function $u(x,y)$ can be written as
	\beq\label{eqn:nabla}\nabla u=h(\xi,\theta)\pd{u}{\xi}\hat{\mathbf{e}}_\xi+h(\xi,\theta)\pd{u}{\theta}\hat{\mathbf{e}}_\theta,
	\quad\mbox{with }\hat{\mathbf{e}}_\xi=\frac{\p \Bx/\p\xi}{|\p \Bx/\p\xi|},\ 
	\hat{\mathbf{e}}_\theta=\frac{\p \Bx/\p\theta}{|\p \Bx/\p\theta|}.
	\eeq

	For any $\xi_0\neq 0$, the level curve $\{\xi=\xi_0\}$ is the circle centered at $\Bc=(c,0)$ with radius $r$ such that
	\beq\label{bipolar:circle}c = \alpha \frac{e^{2\xi_0}+1}{e^{2\xi_0}-1},\quad r=\frac{\alpha }{\sinh \xi_0}.\eeq
	From \eqnref{bipolar:z}, the center $\Bc$ has the bipolar coordinates $(2\xi_0,\pi)$.
	On this circle, the normal and tangential derivatives of a function $u$ satisfy that
	\begin{align}\label{nortan}
	\frac{\p u}{\p \nu} \Big|_{\xi = \xi_0} & = -\mbox{sgn} (\xi_0) h(\xi_0,\theta) \frac{\p u}{\p \xi} \Big|_{\xi = \xi_0},\\\notag
	\frac{\p u}{\p T} \Big|_{\xi = \xi_0} &= -\mbox{sgn} (\xi_0) h(\xi_0,\theta) \frac{\p u}{\p \theta} \Big|_{\xi = \xi_0},
	\end{align}
	where $\nu$ is the outward unit normal vector and $T$ is the rotation of $\nu$ by $\frac{\pi}{2}$-radian angle. 
	We may simply write $\frac{\p u}{\p\nu}(\xi_0,\theta)$ for $\frac{\p u}{\p\nu}|_{\xi=\xi_0}(\xi_0,\theta)$.
	The reflection across the circle $\{\xi=\xi_0\}$ is actually a reflection with respect to $\xi$-variable, as follows.
	\begin{lemma}\label{rbip} 
		For a disk $D$ whose boundary is the coordinate level curve $\{\xi=\xi_0\}$ with some $\xi_0\neq 0$, we have 
		$$R_{\p D}(\Bx(\xi, \theta)) = \Bx(2\xi_0 - \xi,\, \theta),\qquad\mbox{for all }(\xi,\theta)\neq(2\xi_0,\pi).$$\end{lemma}
	\begin{proof}
		Assume that $D$ has the center $\Bc=(c,0)$ and the radius $r$. From \eqnref{bipolar:circle}, we deduce that $\Bc=\Bx(2\xi_0,\pi)$ and 
		$$(c+\alpha)(c-\alpha)=\frac{\alpha^2}{\sinh^2\xi_0}=r^2,\quad e^{2\xi_0}=\frac{c+\alpha}{c-\alpha}.$$
		For $z=x+\rm{i}y\neq c+\rm{i}0$, we have
		\begin{align*}
		\frac{\alpha + R_{\p D}(z)}{\alpha - R_{\p D}(z)}& = \frac{\alpha + \frac{r^2(z-c)}{|z-c|^2} + c}{\alpha - \frac{r^2(z-c)}{|z-c|^2} - c } =   - \frac{c+\alpha}{c-\alpha}\, \frac{\frac{(z-c)(c-\alpha)}{|z-c|^2} +1}{\frac{(z-c)(c+\alpha)}{|z-c|^2} +1} \\
		& = -\frac{c+\alpha}{c-\alpha}\, \frac{(c-\alpha) + (\bar{z} -c)}{(c+\alpha) + (\bar{z}-c)}
		=  \frac{c+\alpha}{c-\alpha}\, \frac{\alpha-\bar{z}}{\alpha+\bar{z}} =e^{2\xi_0 - \xi +{\rm i}\theta}.
		\end{align*}
		We prove the lemma by use of \eqnref{def:bipolar}.
		\qed \end{proof}

	\section{Decomposition of the solution into singular and regular terms}\label{sec:series_bipolar}
	The gradient blow-up of the solution to \eqnref{Eqn} may or may not blow up depending on the applied field $H$. In this section, we show that only the linear term of $H$ contributes to the blow-up feature. This property naturally leads to decomposition of the solution into a singular term (that blows up as $\ep$ tends to zero) and a regular term (that remains bounded independently of $\ep$); see Theorem \ref{main} for further details. 
	
	\subsection{Layer potential formulation}\label{sec:layer}
	
	Let $D$ be a bounded domain in $\RR^2$ with a Lipschitz boundary $\p D$. We define the single layer potential $\Scal_{\p D}$ and the Neumann-Poincar\'{e} operator $\Kcal_{\p D}^*$ as
	\begin{align*}
	\ds\Scal_{\p D }[\varphi] (\Bx) &= \frac{1}{2\pi} \int_{\p D} \ln |\Bx-\By| \varphi (\By)\, d\sigma (\By), \qquad \Bx \in \mathbb{R}^2, \\
	\ds\Kcal_{\p D } ^* [\varphi](\Bx)&= \frac{1}{2\pi}\, p.v. \int_{\p D} \frac{\langle \Bx-\By, \nu_\Bx \rangle}{|\Bx-\By|^2} \varphi(\By)\, d\sigma (\By), \qquad \Bx\in \p D,
	\end{align*}
	where $\varphi \in L^2 (\p D)$. Here {\it p.v.} means the Cauchy principal value.
	They satisfy the relation
	$$\frac{\p}{\p \nu} \Scal_{\p D} [\varphi]\Big|^\pm = \left( \pm \frac{1}{2} + \Kcal_{\p D} ^* \right) \varphi\quad \mbox{on }\p D,$$ 
	where $\nu$ is the outward normal vector on $\p D$.
	For a $\mathcal{C}^{1,\alpha}$-domain $D$, the operator $\mathcal{K}^*_{\p D}$ is compact on $L^2(\p D)$ and the eigenvalues of $\Kcal^*_{\p D}$ on $L^2(\p D)$ lie in $(-1/2, 1/2]$; see \cite{kellog}. The operator $\lambda I-\Kcal^*_{\p D}$ is invertible on $L^2_0(\p\Om)$ for $|\lambda|\geq 1/2$ \cite{book}. For more properties of the Neumann-Poincar\'{e}, see \cite{book2,ammari2007polarization}.

	The solution $u$ to the conductivity equation \eqnref{Eqn} satisfies the transmission conditions 
	\begin{align*}
	\pd{u}{\nu}\Big|^+=k_e \pd{u}{\nu}\Big|^-\qquad&\mbox{on }\p B_e,\\
	k_e\pd{u}{\nu}\Big|^+=k_i \pd{u}{\nu}\Big|^-\qquad&\mbox{on }\p B_i,
	\end{align*}
	and it can be expressed as
	\begin{equation}\label{eqn:BIexpression}
	u(\bold{x}) = H(\bold{x}) + \Scal_{\p B_i} [\varphi_i](\Bx) + \Scal_{\p B_e} [\varphi_e](\Bx),
	\end{equation}
	where $(\varphi_i,\varphi_e)$ are mean-zero functions satisfying
	\begin{equation} 
	\label{eqn:integraleqn0}
	\left(\begin{array}  {cc}
	\ds\lambda_i I-\Kcal_{\p B_i}^* & \displaystyle -\frac{\p \Scal_{\p B_e}}{\p \nu_i} \\[2mm]
	\displaystyle  -\frac{\p \Scal_{\p B_i}}{\p \nu_e} & \lambda_e I-\Kcal^*_{\p B_e}\end{array} \right)
	\left( \begin{array}{c} \varphi_i \\ \varphi_e \end{array} \right)= \left(\begin{array} {c} \displaystyle \frac{\p H}{\p \nu} \Bigr|_{\p B_i} \\[3mm] \displaystyle\frac{\p H}{\p \nu} \Bigr|_{\p B_e} \end{array} \right)
	\end{equation}
	with
	\beq\label{lambda_ie}
	\lambda_i = \frac{k_i + k_e}{2(k_i - k_e)},\quad\displaystyle \lambda_e = \frac{k_e+1}{2(k_e-1)}.
	\eeq 
	Here, $\nu_i$ and $\nu_e$ denote the outward unit normal vectors on $\p B_i$ and $\p B_e$, respectively. Note that $|\lambda_i|,|\lambda_e|>\frac{1}{2}$ due to the positive assumption on the conductivities $k_i,k_e$.
	The problem \eqnref{eqn:BIexpression} is solvable for $|\lambda_i|,|\lambda_e|>\frac{1}{2}$ as shown in \cite{ammari2013spectral}.
	Since $B_i$ and $B_e$ are disks, $\Kcal^*_{\p B_i}$ and $\Kcal^*_{\p B_e}$ are zero operators on $L^2_0(\p B_i)$ and $L^2_0(\p B_e)$, respectively. So, the system of boundary integral equations \eqnref{eqn:integraleqn0} becomes
	\begin{equation}
	\label{eqn:intdisks}
	\left(\begin{array}  {cc}
	\ds\lambda_i I & \displaystyle -\frac{\p \Scal_{\p B_e}}{\p \nu_i} \\[2mm]
	\displaystyle - \frac{\p \Scal_{\p B_i}}{\p \nu_e} & \lambda_e I\end{array} \right)
	\left( \begin{array}{c} \varphi_i \\ \varphi_e \end{array} \right)= \left(\begin{array} {c} \displaystyle ~\frac{\p H}{\p \nu} \Bigr|_{\p B_i} \\[3mm] \displaystyle \frac{\p H}{\p \nu} \Bigr|_{\p B_e} \end{array} \right).
	\end{equation}

	We can relate the tangential component and the normal component of the electric field on $\p B_i$ and $\p B_e$. The following lemma will be used later. One can prove the lemma by slightly modifying the proof of \cite[Theorem 2]{AKL}. 
	\begin{lemma}\label{lemma:tangential}Let $u$ be the solution to \eqnref{Eqn} with an entire harmonic function $H$. Let $\widetilde{H}$ be the harmonic conjugate of $H$ and $\tilde{\sigma}= \frac{1}{k_i} \chi(B_i) + \frac{1}{k_e} \chi(B_e \setminus B_i) + \chi(\mathbb{R}^2 \setminus B_e)$. Then, the solution $v$ to the problem
		\begin{equation*} 
		\left\{ \begin{array}{ll}
		\displaystyle \nabla \cdot \tilde{\sigma} \nabla v = 0 \qquad& \textrm{in } \mathbb{R}^2, \\
		\displaystyle v(\bold{x}) - \widetilde{H}(\bold{x}) =O(|\bold{x}|^{-1})\qquad & \textrm{as } |\bold{x}| \to \infty
		\end{array} \right.
		\end{equation*}
		satisfies that
		$$ \frac{\p u}{\p T} = - \frac{\p v}{\p \nu} \Big|^{+} \textrm { on } \p B_e, ~~~
		\frac{\p u}{\p T} = -\frac{1}{ k_e} \frac{\p v}{\p \nu}\Big|^{+} \textrm { on } \p B_i.
		$$
	\end{lemma}

	\subsection{Series expansions for the density functions in bipolar coordinates}\label{sec:densities}
	
	The single layer potential associated with a disk can be written in terms of the reflection with respect to the disk's boundary.

	\begin{lemma}[\cite{AKL}]\label{SLDisk}Let $D$ be a disk centered at $\Bc$ and $\nu$ be the outward unit normal vector on $\p D$. If $v$ is harmonic in $D$ and continuous on $\overline{D}$, then we have
		$$\Scal_{\p D} \left[\frac{\p v}{\p \nu} \Big|_{\p D}^- \right] (\bold{x}) = 
		\begin{cases}
		\displaystyle -\frac{1}{2} v(\bold{x}) + \frac{v(\Bc)}{2} \quad& \mbox{for }\bold{x} \in D,\\[3mm]
		\displaystyle -\frac{1}{2}  R_{\p D} [v ](\bold{x}) + \frac{v(\Bc)}{2} \quad&\mbox{for } \bold{x} \in \mathbb{R}^2 \setminus \overline{D}.
		\end{cases}
		$$
		If $v$ is harmonic in $\mathbb{R}^2 \setminus \overline{D}$, continuous on $\mathbb{R}^2 \setminus D$ and $v(\bold{x}) \to 0$ as $|\bold{x}| \to \infty$, then we have
		$$\Scal_{\p D} \left[\frac{\p v}{\p \nu} \Big|_{\p D}^+ \right] (\bold{x}) = 
		\left\{ \begin{array}{ll}
		\ds\frac{1}{2}  R_{\p D}[ v ] (\bold{x})\quad&\mbox{for } \bold{x} \in D,\\[3mm]
		\ds\frac{1}{2} v(\bold{x}) \quad& \mbox{for }\bold{x} \in \mathbb{R}^2 \setminus \overline{D}.
		\end{array} \right.
		$$
	\end{lemma}
	%
	%

	We now express the density functions $\varphi_i$ and $\varphi_e$ in \eqnref{eqn:BIexpression} in series form by use of Lemma \ref{rbip}.
	\begin{lemma}\label{lemma:density}
		The density functions in \eqnref{eqn:BIexpression} can be expressed as
		\begin{equation}
		\displaystyle \label{expan:phii}
		\begin{cases}
		\ds\lambda_i\varphi_i(\theta) = \left(1-\frac{1}{2\lambda_e}\right) \sum_{m=0} ^{\infty} \left(-\frac{1}{4\lambda_i \lambda_e } \right)^m \frac{h(\xi_i,\theta)}{h(\xi_{i,m},\theta)}\frac{\p H}{\p\nu}(\xi_{i,m},\theta),\\[3mm]
		\displaystyle 
		\lambda_e\varphi_e(\theta) = \frac{\p H}{\p \nu}(\xi_e,\theta) +\frac{1}{2\lambda_i} \left(1-\frac{1}{2\lambda_e}\right) \sum_{m=0} ^{\infty} \left(-\frac{1}{4\lambda_i \lambda_e } \right)^m  \frac{h(\xi_e,\theta)}{h(\xi_{e,m+1},\theta)}\frac{\p H}{\p \nu} (\xi_{e,m+1},\theta).
		\end{cases}
		\end{equation}
	\end{lemma}
	\begin{proof}
		First, we formally expand the solution $(\varphi_i,\varphi_e)$ to \eqnref{eqn:intdisks} as follows (uniform and absolute convergence will be proved later):
		\begin{equation}
		\label{phii:reflection}
		\begin{cases}
		\ds\lambda_i\varphi_i =  \left(1-\frac{1}{2\lambda_e}\right)\sum_{m=0} ^{\infty} \left(-\frac{1}{4\lambda_i \lambda_e } \right)^m \frac{\p}{\p \nu_i}\big( (R_e R_i)^m H \big)\Big|_{\p B_i}, \\[3mm] 
		\ds\lambda_e \varphi_e =\frac{\p H}{\p \nu_e} \Big|_{\p B_e} - \frac{1}{2\lambda_i} \left(1-\frac{1}{2\lambda_e}\right)\sum_{m=0} ^{\infty} \left(-\frac{1}{4\lambda_i \lambda_e } \right)^m \frac{\p}{\p \nu_e} \big((R_i R_e)^m R_i H \big)\Big|_{\p B_e}.
		\end{cases}
		\end{equation}
		Using Lemma \ref{SLDisk}, we compute
		\begin{align*}
		\ds\frac{\p \Scal_{\p B_i} [\varphi_i]}{\p \nu_e} 
		& = \frac{1}{\lambda_i}  \left(1-\frac{1}{2\lambda_e}\right)\sum_{m=0} ^{\infty} \left(-\frac{1}{4\lambda_i \lambda_e } \right)^m \frac{\p}{\p \nu_e} \Scal_{\p B_i}\left[ \frac{\p}{\p \nu_i} \big((R_e R_i)^m H \big)\Big|_{\p B_i}\right] \\
		\ds&= -\frac{1}{2 \lambda_i}  \left(1-\frac{1}{2\lambda_e}\right)\sum_{m=0} ^{\infty} \left(-\frac{1}{4\lambda_i \lambda_e } \right)^m \frac{\p}{\p \nu_e} \big(R_i(R_e R_i)^m H\big) \Big|_{\p B_i}\\
		\ds&= \lambda_e \varphi_e - \frac{\p H}{\p \nu_e}\Big|_{\p B_e},
		\end{align*}
		and
		\begin{align*}
		\ds\frac{\p \Scal_{\p B_e} [\varphi_e]}{\p \nu_i} & = \frac{1}{\lambda_e} \frac{\p}{\p \nu_i} \Scal_{\p B_e} \left[\frac{\p H}{\p \nu_e} \Big|_{\p B_e} \right] \\&
		\quad- \frac{1}{2\lambda_i \lambda_e} \left(1-\frac{1}{2\lambda_e}\right)\sum_{m=0} ^{\infty} \left(-\frac{1}{4\lambda_i \lambda_e } \right)^m \frac{\p}{\p \nu_i} \Scal_{\p B_e} \left[ \frac{\p}{\p \nu_e}\big((R_i R_e)^m R_i H\big) \Big|_{\p B_e}\right] \\
		\ds&= -\frac{1}{2\lambda_e} \frac{\p H}{\p \nu_i} \Big|_{\p B_i} - \frac{1}{4\lambda_i \lambda_e} \left(1-\frac{1}{2\lambda_e}\right)\sum_{m=0} ^{\infty} \left(-\frac{1}{4\lambda_i \lambda_e } \right)^m \frac{\p}{\p \nu_i}\big((R_i R_e)^{m+1} H\big) \Big|_{\p B_i}\\
		\ds&= \lambda_i \varphi_i - \frac{\p H}{\p \nu_i}\Big|_{\p B_i}.
		\end{align*}
		Therefore the series solutions in \eqnref{phii:reflection} satisfy \eqnref{eqn:intdisks}.

		For the entire harmonic function $H$, we have from Lemma \ref{rbip} that
		\begin{align*}
		\ds R_e R_i H(\xi,\theta) &= H(R_i \circ R_e(\xi,\theta)) = H(2(\xi_i - \xi_e)+\xi,\theta), \\
		\ds R_i R_e R_i H(\xi,\theta) &= H(R_i \circ R_e\circ R_i(\xi,\theta)) = H(2(\xi_i - \xi_e)+2\xi_i-\xi,\theta).
		\end{align*}
		Repeating the reflections, it follows that
		\begin{align*}
		\ds (R_e R_i)^m H(\xi,\theta)&= H(2m(\xi_i-\xi_e)+\xi,\theta),\\
		\ds (R_i R_e)^m R_i H(\xi,\theta) &= H(2m(\xi_i-\xi_e)+2\xi_i-\xi,\theta).
		\end{align*}
		Recall that $\xi_{i,m}$ and $\xi_{e,m+1}$ are given as \eqnref{def:xik}. By use of \eqnref{nortan} we deduce that
		\begin{align*}
		\frac{\p}{\p \nu_i}\big( (R_e R_i)^m H \big)\Big|_{\p B_i}& = -h(\xi_i,\theta)\frac{\p}{\p \xi}\left(H\left(2m(\xi_i-\xi_e)+\xi,\theta\right)\right)\Big|_{\xi=\xi_i} \\
		&= \frac{h(\xi_i,\theta)}{h(\xi_{i,m},\theta)}\frac{\p H}{\p\nu} (\xi_{i,m},\theta),\\[1mm]
		\frac{\p}{\p \nu_e}\big((R_i R_e)^m R_i H \big)\Big|_{\p B_e} &= -h(\xi_e,\theta)\frac{\p}{\p \xi}\left(H\left(2m(\xi_i-\xi_e)+2\xi_i-\xi,\theta\right)\right)\Big|_{\xi=\xi_e} \\
		&= -\frac{h(\xi_e,\theta)}{h(\xi_{e,m+1},\theta)}\frac{\p H}{\p \nu} (\xi_{e,m+1},\theta).
		\end{align*}
		Therefore, the equation \eqnref{phii:reflection} becomes \eqnref{expan:phii}. 
		
		Now, we show uniform and absolute convergence for the series expansions. 
		Recall that $\xi_{i,m}\geq\xi_i$ and $\xi_{e,m+1}\geq \xi_e$. 
		Since $\Bx(\xi_{i,m},\theta),\Bx(\xi_{e,m+1},\theta)\in B_e$ for all $m$ and 
		$$ \frac{h(\xi_i,\theta)}{h(\xi_{i,m},\theta)}\leq 1,\quad \frac{h(\xi_e,\theta)}{h(\xi_{e,m+1},\theta)}\leq1,$$ 
		we have
		\begin{align*}
		\frac{h(\xi_i,\theta)}{h(\xi_{i,m},\theta)} \left|\frac{\p H}{\p \nu}(\xi_{i,m},\theta)\right| &\le \Vert\nabla H(\bold{x})\Vert_{L^\infty(B_e)},\\[1mm]
		\frac{h(\xi_e,\theta)}{h(\xi_{e,m+1},\theta)} \left|\frac{\p H}{\p \nu}(\xi_{e,m+1},\theta)\right|& \le \Vert\nabla H(\bold{x})\Vert_{L^\infty(B_e)}.
		\end{align*}
		Note that
		$$\left| \frac{1}{4\lambda_i \lambda_e } \right| < 1$$ due to $|\lambda_i|,|\lambda_e|>\frac{1}{2}$. 
		Therefore, each term in \eqnref{phii:reflection} (in other words, each term in \eqnref{expan:phii}) is bounded by $\Vert \nabla H(\bold{x})\Vert_{L^\infty(B_e)}\left| \frac{1}{4\lambda_i \lambda_e } \right|^m$. This implies convergence and so the proof is completed.
		 \end{proof}

	\begin{cor}\label{boundedness1}
		If $k_i$ and $k_e$ are bounded but not small, then $|\nabla(u-H)|$ is uniformly bounded in $\mathbb{R}^2$ independently of $\ep$.
	\end{cor}
	\begin{proof}
		From Lemma \ref{lemma:density} and the analysis in its proof, we have
		$$\|\varphi_i\|_{L^\infty(\p B_i)}, \|\varphi_e\|_{L^\infty(\p B_e)} \leq 6\frac{\|\nabla H\|_{L^\infty(B_e)}}{1-\frac{1}{|4\lambda_i\lambda_e|}}.$$
		The jump formula for the single layer potential and \eqnref{eqn:intdisks} implies 
		\beq\label{jump}
		\left\Vert \frac{\p u}{\p \nu} \Big|^{\pm} \right\Vert_{L^\infty (\p B_i)} = \left|\lambda_i \pm \frac{1}{2} \right| \Vert \varphi_i \Vert_{L^\infty (\p B_i)}.
		\eeq
		Therefore we have uniform boundedness for normal derivatives of $u$ on $\p B_i$ and $\p B_e$. 
		We can also prove uniform boundedness for tangential derivatives by using a harmonic conjugate of $u$, in view of Lemma \ref{lemma:tangential}.
		
		Since $(u-H)(\Bx)$ satisfies the decay condition as $|\Bx|\rightarrow\infty$, $|\nabla(u-H)|$ is uniformly bounded in $\mathbb{R}^2$ independently of $\ep$ by the maximum principle.
		 \end{proof}
	
	\subsection{Blow-up contribution term in $H$ and decomposition of the solution}\label{subsec:decomposition}
	
	We prove the following upper bounds of two series sums related with the scale factor function $h$, which is essential for showing that only the linear term of $H$ contributes to the gradient blow-up.
	\begin{lemma} \label{coshsum} 
		Let $N_\ep = \frac{r_*}{\sqrt{\epsilon}}$, then we have  
		\begin{equation*}
		\begin{cases}
		\ds\sum_{m=0}^{N_\ep} \frac{h(\xi_i,\theta)}{h(\xi_{i,m},\theta)}
		\le  \frac{C}{|z(\xi_i,\theta)|} \qquad\mbox{for all }|\theta|>\frac{\pi}{2},\\[4mm]
		\ds\sum_{m=0}^{N_\ep} \frac{h(\xi_e,\theta)}{h(\xi_{e,m+1},\theta)}
		\le  \frac{C}{|z(\xi_e,\theta)|} \qquad\mbox{for all }|\theta|>\frac{\pi}{2},
		\end{cases}
		\end{equation*}
		where $C$ is a constant independent of $\ep,\,\theta$.
	\end{lemma}
	\begin{proof}
		We prove the first inequality only. The second one can be shown in the same way.
		
		Note that 
		\begin{align}
		\frac{h(\xi_i,\theta)}{h(\xi_{i,m},\theta)}&=\frac{\cosh \xi_i + \cos \theta}{\cosh \xi_{i,m} + \cos \theta} = \frac{(\cosh \xi_i + \cos \theta)}{(\cosh \xi_i - \cos \theta)}\frac{(\cosh \xi_{i,m} - \cos \theta)}{(\cosh \xi_{i,m} + \cos \theta)}\frac{(\cosh \xi_i - \cos \theta)}{(\cosh \xi_{i,m} - \cos \theta)} \notag\\\label{eqn:ratio}
		& \le  \frac{(\cosh \xi_i + \cos \theta)}{\alpha^2 (\cosh \xi_i - \cos \theta)}\frac{\alpha^2 (\cosh \xi_{i,m} - \cos \theta)}{(\cosh \xi_{i,m} + \cos \theta)}  = \left|\frac{z(\xi_{i,m},\theta)}{z(\xi_i,\theta)}\right|^2 .
		\end{align}
		
		We use \eqnref{def:bipolar}, to estimate the last term in the above equation, and obtain
		$$\frac{\alpha+z(\xi_{i,m},\theta)}{\alpha-z(\xi_{i,m},\theta)} =e^{\xi_{i,m}+\rm{i}\theta}=e^{2m(\xi_i-\xi_e)}e^{\xi_i+\rm{i}\theta}=\gamma^m\frac{\alpha+z(\xi_i,\theta) }{\alpha-z(\xi_i,\theta) }$$
		with $\gamma =e^{2(\xi_i - \xi_e)}$.
		This implies 
		\begin{align*}
		z(\xi_{i,m},\theta) 
		&=\alpha\frac{z(\xi_i,\theta)(\gamma^m+1)+\alpha(\gamma^m-1)}{z(\xi_i,\theta)(\gamma^m-1)+\alpha(\gamma^m+1)}= \frac{\displaystyle 1+\frac{\alpha}{z(\xi_i,\theta)}\frac{\gamma^m - 1}{\gamma^m + 1}}{\displaystyle \frac{1}{\alpha}\frac{\gamma^m -1}{\gamma^m +1} + \frac{1}{z(\xi_i,\theta)}}.
		\end{align*}
		From \eqnref{def:rstar} and \eqnref{bipolar:circle}, we deduce $\xi_i = \frac{r_*}{r_i}\sqrt{\ep}+O(\ep\sqrt{\ep})$ and $\xi_e = \frac{r_*}{r_e}\sqrt{\ep}+O(\ep\sqrt{\ep})$ so that 
		\beq\label{xidiff}
		\xi_i-\xi_e=\frac{2}{r_*}\sqrt{\ep}+O(\ep\sqrt{\ep})
		\eeq
		and
		$$1+\frac{3}{r_*} \sqrt{\epsilon} < \gamma = e^{2(\xi_i - \xi_e)}< 1+\frac{5}{r_*} \sqrt{\epsilon}.$$
		Hence, by using \eqnref{thetalevel} as well, we have for $1 \le m < \frac{r_*}{\sqrt{\epsilon}}$ and $|\theta|>\frac{\pi}{2}$  that
		$$1+\frac{2m}{r_*}\sqrt{\epsilon}<\gamma^m < \left( 1+\frac{5}{r_*}\sqrt{\epsilon}\right)^{ \frac{r_*}{\sqrt{\epsilon}}}< e^5$$
		and $$|z(\xi_{i,m},\theta)| \leq\frac{1+\frac{\alpha}{|z(\xi_i,\theta)|}\frac{ e^5-1}{2}}{\left|\frac{2m\sqrt{\epsilon}/r_*}{\alpha(e^5+1)} + \frac{1}{z(\xi_i,\theta)} \right|}
		\leq \frac{C_1}{C_2 m+ \frac{1}{|z(\xi_i,\theta)|}}$$
		with some positive constants $C_1$ and $C_2$ independent of $\ep$, $\theta$ and $m$. 
		Hence, we have
		\beq\label{ineq:zxi}
		\left|\frac{z(\xi_{i,m},\theta)}{z(\xi_i,\theta)}\right|\leq \frac{C}{ m|z(\xi_i,\theta)| + 1}\qquad \mbox{for }1 \le m < \frac{r_*}{\sqrt{\epsilon}},\  |\theta|>\frac{\pi}{2}.
		\eeq
		
		Note that for any $t>0$, the following holds: 
		\begin{align*}
		\sum_{m=0}^\infty \frac{1}{ (m t + 1)^2} 
		&\le 1+\int_0^\infty\frac{1}{(xt+1)^2}\,dx=1+\frac{1}{t}.
		\end{align*} 
		This completes the proof because of \eqnref{eqn:ratio}, \eqnref{ineq:zxi} and the boundedness of $|z(\xi_i,\theta)|$.
		 \end{proof}
	
	\smallskip
	
	Now, we show that $|\nabla (u-H)|$ is bounded regardless of $\epsilon$ if $\nabla H(\Bx)=O(|\Bx|)$ near the origin.
	\begin{lemma}\label{lemma:boundedness} If $\nabla H(\Bx)=O(|\Bx|)$ in a bounded region containing $B_e$, then we have
		$$\lVert \nabla (u-H)\rVert_{L^\infty (\mathbb{R}^2)} \le C$$
		for some constant $C$ independent of $k_i,\, k_e$ and $\ep$.
	\end{lemma}
	\begin{proof}
		Let $\nabla H$ denote the gradient with respect to the Cartesian coordinates, and $\xi_{i,m},\xi_{e,m+1}$ are defined as \eqnref{def:xik}. 
		
		We first show that
		\beq\label{IandII}\sum_{m=0} ^{\infty} \frac{h(\xi_i,\theta)}{h(\xi_{i,m},\theta)}  \big| \nabla H(\xi_{i,m},\theta)\big|\leq C.
		\eeq
		Here and after, $C$ denotes a constant independent of $\ep$, $\theta$ and $m$.
		For notational simplicity we decompose the series in \eqnref{IandII} as
		$$\sum_{m=0} ^{\infty} \frac{h(\xi_i,\theta)}{h(\xi_{i,m},\theta)}  \big| \nabla H(\xi_{i,m},\theta)\big| = \sum_{m \le \frac{r_*}{\sqrt{\ep}}} + \sum_{m >\frac{r_*}{\sqrt{\ep}}}:= ~  I + II$$
		and separately prove the uniform boundedness.
		For $|\theta|\leq\frac{\pi}{2}$, we have from \eqnref{thetalevel} and the assumption $\nabla H(\Bx)=O(|\Bx|)$ in $B_e$ that
		$$I\leq C\frac{r_*}{\sqrt{\ep}}\big|z(\xi_{i,m},\theta)\big|\leq C\frac{r_*}{\sqrt{\ep}}\alpha=O(1).$$
		For $|\theta|>\frac{\pi}{2}$, we have from Lemma \ref{coshsum} and \eqnref{ineq:zxi}
		$$I\leq \frac{C}{|z(\xi_i,\theta)|}\max\left\{\big|z(\xi_{i,m},\theta)\big|:m \le \frac{r_*}{\sqrt{\ep}}\right\}=O(1).$$
		We now estimate $II$. 
		For $m>\frac{r_*}{\sqrt{\ep}}$, we have $\xi_{i,m}>1$ due to \eqnref{xidiff} so that $$\big|z(\xi_{i,m},\theta)\big| < C\alpha$$
		and
		$$\frac{h(\xi_i,\theta)}{h(\xi_{i,m},\theta)}=\frac{\cosh\xi_i+\cos\theta}{\cosh \xi_{i,m}+\cos\theta}
		\leq \frac{\cosh\xi_i}{\frac{1}{2}\cosh\xi_{i,m}}\leq 4 e^{\xi_i-\xi_{i,m}}=4e^{-2m(\xi_i-\xi_e)}.$$
		Therefore, it follows that
		$$
		II
		\leq \sum_{m>\frac{r_*}{\sqrt{\ep}}}4e^{-2m(\xi_i-\xi_e)}C\big|z(\xi_{i,m},\theta)\big|\leq \frac{C\alpha}{\xi_i-\xi_e}= O(1).
		$$
		In the same way, one can easily show that
		\beq\label{IandII2}
		\sum_{m=0} ^{\infty} \frac{h(\xi_i,\theta)}{h(\xi_{e,m+1},\theta)}  \left| \nabla H(\xi_{e,m+1},\theta)\right|\leq C.
		\eeq

		From \eqnref{IandII}, \eqnref{IandII2} and Lemma \ref{lemma:density} we obtain
		\begin{align*}
		|\varphi_i| < \frac{C}{\lambda_i} \left( 1-\frac{1}{2\lambda_e} \right), \quad |\varphi_e| < \frac{C}{\lambda_e}
		\end{align*}
		and, thus,
		\begin{align*}
		\left\Vert \frac{\p u}{\p \nu} \bigg|_{\pm} \right\Vert_{L^\infty (\p B_i)} &= \left|\lambda_i \pm \frac{1}{2} \right| \Vert \varphi_i \Vert_{L^\infty (\p B_i)} \le C.
		\end{align*}
		The same result holds on $\p B_e$.

		We can also prove uniform boundedness for tangential derivatives by using a harmonic conjugate for the harmonic function; see Lemma \ref{lemma:tangential}.
		Following the same argument in the proof of cor \ref{boundedness1}, we have uniform boundedness for $|\nabla(u-H)|$ in $\mathbb{R}^2$ independently of $\ep$.
		 \end{proof}

	As a direct consequence of Lemma \ref{lemma:boundedness}, we can decompose the solution into a singular term and a regular term as follows.
	\begin{theorem}\label{main}
		Let $H$ be an arbitrary entire harmonic function. Set $H_1(\Bx) = \nabla H(\mathbf{0})\cdot \Bx$ and $H_2(\Bx)=H(\Bx)-H_1(\Bx)$. Then, the solution $u$ to the transmission problem \eqnref{Eqn} has the decomposition
		\beq
		u(\Bx)= u_1(\Bx)+ u_2(\Bx),\quad\Bx\in\RR^2,
		\eeq
		where $u_1$ and $u_2$ are the solutions to \eqnref{Eqn} with $H_1$ and $H_2$ in the place of $H$, respectively, and $\Vert \nabla (u_2 -H_2)(\Bx)\Vert_{L^\infty(\RR^2)} < C$ for some constant $C$ independent of $\epsilon$.
	\end{theorem}
	\begin{cor}
		If $\lambda_i\lambda_e>\frac{1}{4}$ (equivalently, either $k_i>k_e>1$ or $k_i<k_e<1$), then for any entire function $H$ we have
		$$\lVert \nabla (u-H)\rVert_{L^\infty (\mathbb{R}^2)} \le C$$
		with some positive constant $C$ independent of $\ep$.
	\end{cor}
	\begin{proof}
		Thanks to Theorem \ref{main}, we may consider only the linear function for $H$, say $H(x,y)=ax+by$. 
		Then, from \eqnref{nortan} we have
		\begin{align*}
		&\sum_{m=0} ^{\infty} \left(-\frac{1}{4\lambda_i \lambda_e } \right)^m \frac{h(\xi_i,\theta)}{h(\xi_{i,m},\theta)}\frac{\p H}{\p\nu} (\xi_{i,m},\theta) \\
		=&-\sum_{m=0} ^{\infty} \left(-\frac{1}{4\lambda_i \lambda_e } \right)^m h(\xi_i,\theta)\left(a\pd{x}{\xi}+b\pd{y}{\xi}\right)\Big|_{\xi=\xi_{i,m}}.
		\end{align*}
		
		Note that from \eqnref{eqn:xdxi}, $\pd{x}{\xi}$, $\pd{y}{\xi}$ and their derivatives in $\xi$ have at most a finite number of changes in sign when $\theta$ is fixed. Moreover, 
		$$h(\xi_i,\theta)\Big|\pd{x}{\xi}(\xi_{i,m},\theta)\Big|, \, h(\xi_i,\theta)\Big|\pd{y}{\xi}(\xi_{i,m},\theta)\Big|\leq \frac{h(\xi_i,\theta)}{h(\xi_{i,m},\theta)}\leq 1.$$ 
		Therefore, from the convergence property of alternating series we derive that
		\beq\notag
		\left| \sum_{m=0} ^{\infty} \left(-\frac{1}{4\lambda_i \lambda_e } \right)^m \frac{h(\xi_i,\theta)}{h(\xi_{i,m},\theta)}\frac{\p H}{\p\nu} (\xi_{i,m},\theta) \right| \le C \left| \nabla H(\bold{0}) \right|
		\eeq
		and, in the same way,
		\beq\notag \left| \sum_{m=0} ^{\infty} \left(-\frac{1}{4\lambda_i \lambda_e } \right)^m \frac{h(\xi_e,\theta)}{h(\xi_{e,m+1},\theta)}\frac{\p H}{\p\nu} (\xi_{e,m+1},\theta) \right| \le C \left| \nabla H(\bold{0}) \right|.
		\eeq
		Hence, $\varphi_i, \varphi_e$ are bounded independent of $\ep$. Following the same argument in cor \ref{boundedness1} we have uniform boundedness for $|\nabla(u-H)|$.
		 \end{proof}

	For the case of separated disks, a boundedness result comparable to Lemma \ref{lemma:boundedness} was derived in \cite{AKLLZ}, based on the series expression in the Cartesian coordinates that is similar to \eqnref{phii:reflection}. We emphasize that  in this paper, we provide a much simpler proof thanks to Lemma \ref{coshsum}.

	\section{Asymptotics of the singular term in bipolar coordinates}\label{sec:bipolarapprox}
	
	From the analysis in the previous section, the gradient blow-up does not occur either when $\nabla H(\Bx) = O(|\Bx|)$ or when $\lambda_i \lambda_e>\frac{1}{4}$ (equivalently, either $k_i>k_e>1$ or $k_i<k_e<1$).
	In this section, we first derive an asymptotic expansion for the solution to \eqnref{Eqn} assuming that $H$ is a linear function and $\lambda_i \lambda_e<-\frac{1}{4}$ (equivalently, either $k_i>k_e,\, 1>k_e$ or $k_i<k_e,\,1<k_e$). We then characterize the gradient blow-up. We follow the derivation in \cite{LYu2D}, where the asymptotic behavior for the electric field was derived in the presence of two separated circular conductors.

	The gradient blow-up feature is essentially related to the following quantities:
	\beq \label{lambdabeta}
	\tau_i=\frac{k_i -k_e}{k_i+k_e},\quad \tau_e = \frac{1-k_e}{1+k_e},\quad \tau = \tau_i\tau_e,\quad \beta = \frac{r_*(-\ln \tau)}{4\sqrt{\epsilon}},
	\eeq
	and 
	\beq\label{eqn:CxCy}
	C_x =  \left( \nabla H(\bold{0}) \cdot \Be_1 \right),\quad  C_y =  \left( \nabla H(\bold{0}) \cdot \Be_2\right),
	\eeq
	where $\Be_1=(1,0)$ and $\Be_2=(0,1)$. 
	Note that $\tau_i = \frac{1}{2\lambda_i},\ \tau_e=-\frac{1}{2\lambda_e}$ and $\beta>0$.

	\subsection{Series solution by the separation of variables}
	In bipolar coordinates, harmonic functions admit the general solution expansion
	$$f(\xi, \theta) = a_0 + b_0 \xi + c_0 \theta + \sum_{n=1} ^{\infty} \left[(a_n e^{n\xi} + b_n e^{-n\xi} )\cos n\theta + (c_n e^{n\xi} + d_n e^{-n\xi}) \sin n\theta\right],$$
	where $a_n, b_n, c_n$ and $d_n$ are some constants. 
	The linear functions $x$ and $y$ can be expanded as
	\begin{align*}
	x &= \mbox{sgn}(\xi) \alpha \left[ 1+ 2 \sum_{n=1} ^{\infty} (-1)^n e^{-n|\xi|}\cos n\theta \right], \\
	y &= -2 \alpha \sum_{n=1} ^{\infty} (-1)^n e^{-n |\xi|} \sin n \theta.
	\end{align*}
	\begin{lemma}\label{lemma:transmissionseries}
		We define
		\begin{equation*}
		U(\bold{x}) := C+ \begin{cases}
		\ds \sum_{n=1} ^{\infty} \left(A_n e^{n(\xi - \rm{i}\theta -2\xi_i )}  + B_n e^{n(\xi - \rm{i}\theta-2\xi_e )} \right)\quad& \mbox{for }\xi<\xi_e,\\[1mm]
		\ds \sum_{n=1} ^{\infty} \left(A_n e^{n(\xi - \rm{i}\theta -2\xi_i )} + B_n e^{n(-\xi - \rm{i}\theta)} \right) \quad& \mbox{for }\xi_e<\xi<\xi_i,\\[1mm]
		\ds \sum_{n=1} ^{\infty} \left(A_n e^{n(- \xi - \rm{i}\theta)} + B_n e^{n(-\xi - \rm{i}\theta)}\right) \quad& \mbox{for }\xi>\xi_i
		\end{cases}
		\end{equation*}
		with
		\begin{align*}
		A_n& =\frac{2\alpha (-1)^{n}}{\tau^{-1} e^{2n(\xi_i - \xi_e)}-1}\left( - 1 -\tau_e^{-1} \right) e^{2n(\xi_i-\xi_e)}, \\
		B_n& =  \frac{2\alpha (-1)^{n}}{\tau^{-1} e^{2n(\xi_i - \xi_e)}-1}\left(1 + \tau_i^{-1}{ e^{2n(\xi_i - \xi_e)}}\right),
		\end{align*}
		and 
		$ C=-\sum_{n=1} ^\infty A_n e^{n(-2\xi_i-{\rm i}\pi)} +  B_n e^{-{\rm i}n\pi}$. 
		Then, the functions $x+\Re\{U(\bold{x})\}$ and $y+\Im\{U(\bold{x})\}$ are the solutions to \eqnref{Eqn} for $H(\bold{x})=x$ and $H(\bold{x})=y$, respectively.
	\end{lemma}
	\begin{proof}
		One can easily show that $x+\Re\{U(\bold{x})\}$ and $y+\Im\{U(\bold{x})\}$ satisfy the transmission condition in the equation \eqnref{Eqn} and demonstrate uniform convergence. We can prove the decay property in the same way as in Lemma 3.1 of \cite{LYu2D}. 
		 \end{proof}

	\begin{definition}
		The Lerch transcendent function is defined as
		\beq\label{eqn:Lerch}
		L(z;\beta)= -\int_0 ^\infty \frac{ze^{-(\beta+1)t}}{1+z e^{-t}} dt \qquad \mbox{for } z\in \mathbb{C}, |z| <1.
		\eeq
		We also define
		$$
		P(z;\beta)= (-z) \frac{\partial}{\partial z} L(z;\beta) = \int_0 ^\infty \frac{ze^{-(\beta+1)t}}{(1+z e^{-t})^2} dt.$$
	\end{definition}
	
	Let us derive some properties on $P$ for later use. 
	By the integration by parts it follows that for all $s>0$,
	\begin{align}
	P(e^{-s-\rm{i}\theta};\beta)
	&=\int_0^\infty e^{-\beta t}\frac{e^{-s-t-\rm{i}\theta}}{(1+e^{-s-t-\rm{i}\theta})^2} \,dt\notag\\\label{P:integral}	&=\int_0^\infty e^{-\beta t}\frac{1+\cosh (s+t)  \cos \theta - \rm{i} \sinh (s+t) \sin \theta}{2(\cosh (s+t) + \cos \theta)^2}\,dt.
	\end{align}
	One can easily show that 
	\begin{align} \label{Pproperty1}
	\left|P(e^{-s+\rm{i}\theta }) \right|&  \le \int_0^\infty e^{-\beta t} \frac{1}{2(\cosh (s+t) + \cos \theta)} \,dt\leq \frac{1}{2\beta(\cosh s + \cos \theta)}.
	\end{align}
	For any $s_2>s_1>0,\, \beta>0$ and $|\theta|\leq\pi$, the following holds:
	\beq\label{Pproperty2}
	\left|P(e^{-s_2+\rm{i}\theta};\beta)-P(e^{-s_1+\rm{i}\theta};\beta)\right|\leq \frac{s_2-s_2}{2(\cosh s_1+\cos\theta)}.
	\eeq
	
	The following integral approximation for an infinite series sum is crucial to obtaining the gradient blow-up term of $u$.
	\begin{lemma}[\cite{LYu2D}]\label{lemma:LYu2D:series} Let $0<a<a_0$ and $0<\tau<1$. For arbitrary $\theta \in (-\pi, \pi]$, we have
		$$\left| a_0 \sum_{m=1} ^{\infty}\left( \tau^{m-1} \frac{e^{-ma_0 + a + \rm{i}\theta}}{(1+e^{-ma_0 + a + \rm{i}\theta})^2}\right) - P\bigr(e^{-(a_0 - a)+\rm{i}\theta};\, \frac{-\ln{\tau}}{a_0}\big) \right| \le \frac{4a_0}{\cosh(a_0-a) + \cos \theta}.$$
	\end{lemma}

	\subsection{Integral expression of the singular term in bipolar coordinates}
	We verify the asymptotic behavior of $u$ in bipolar coordinates as follows.
	We will later express the gradient blow-up term of $u$ in terms of the Cartesian coordinates in section \ref{sec:cartesian}.
	\begin{theorem} \label{blowupterm}
		Let $u$ be the solution to \eqnref{Eqn}. Then we have the following:
		\begin{itemize}
			\item[\rm(a)]
			The solution $u$ admits that
			\beq\label{u:singular}
			u(\Bx) = H(\Bx) + C_x r_*^2\Re\big\{q\big(\Bx;\beta,\tau_i, \tau_e\big)\big\} + C_y r_*^2\Im\big\{q\big(\Bx;\beta,\tau_i, \tau_e\big)\big\} + r(\Bx),\quad \Bx\in\RR^2,
			\eeq
			where {\it the singular function} $q$ is
			\begin{align*} \label{eqn:asymptotic}
			&q(\bold{x};\beta,\tau_i, \tau_e) \\
			&:= \frac{1}{2}
			\begin{cases}
			\ds-\big(\tau+ \tau_i\big)L(e^{\xi-\rm{i}\theta-2\xi_i} ;\beta)+ \big(\tau+\tau_e\big)L(e^{\xi-\rm{i}\theta-2\xi_e};\beta)   \quad& \textrm{in } \mathbb{R}^2 \setminus B_e, \\[.5mm]
			\ds- \big(\tau+ \tau_i\big)L(e^{\xi-\rm{i}\theta-2\xi_i} ;\beta)+ \big(\tau+\tau_e\big)L(e^{-\xi-\rm{i}\theta};\beta)  \quad& \textrm{in } B_e \setminus B_i, \\[.5mm]
			\ds- \big(\tau+ \tau_i\big)L(e^{-\xi-\rm{i}\theta} ;\beta) + \big(\tau+\tau_e\big)L(e^{-\xi-\rm{i}\theta};\beta)  \quad& \textrm{in } B_i,
			\end{cases}
			\end{align*}
			and $\Vert \nabla r(\Bx) \Vert_{L^\infty(\RR^2)} < C$ for some constant $C$ independent of $\epsilon$.\\
			
			\item[\rm(b)]
			For any $k_i,k_e$, we have that $\|\nabla (u-H)\|_{L^\infty(\RR^2\setminus\overline{(B_e\setminus B_i)})}\leq C$ for some constant $C$ independent of $\ep$.
			The gradient blow-up occurs only when $\tau_i,\tau_e$ are similar to $1$ (equivalently, $0<k_e\ll 1,k_i$). For such a case, we have that in $B_e\setminus B_i$,
			\begin{align*}
			&\nabla (u-H)(\bold{x})\\
			&= 
			\ds \frac{C_x r_*(\tau_i+\tau_e+2\tau)}{2 \sqrt{\epsilon}}\big(\cosh \xi + \cos\theta\big) \Re\left\{P(e^{-\xi - \rm{i}\theta};\beta) \right\}\hat{\bold{e}}_\xi \\ \ds \quad &+\frac{C_y r_*( \tau_i+\tau_e+2\tau )}{2\sqrt{\epsilon}}\big(\cosh \xi + \cos\theta\big) \Im\left\{ P(e^{-\xi - \rm{i}\theta};\beta) \right\}\hat{\bold{e}}_\xi + O(1),
			\end{align*}
			where $O(1)$ is uniformly bounded independently of $\epsilon$ and $\Bx$.
		\end{itemize}
	\end{theorem}

	\begin{proof}
		In view of Theorem \ref{main}, we may consider only the linear function for $H$. In other words, $H(\Bx)=(\nabla H(\bold{0}) \cdot (1,0)) x +(\nabla H(\bold{0}) \cdot (0,1)) y$. Moreover, we can assume $\tau=\tau_i\tau_e>0$ since the blow-up happens only in such a case.

		One can derive a series expansion for $\nabla U(\Bx)$ by differentiating the formula in Lemma \ref{lemma:transmissionseries}.
		From the definition of $P$, we have
		\begin{align*}
		\frac{\partial q}{\partial \xi}& = \frac{1}{2} \begin{cases}
		(\tau+ \tau_i )P(e^{\xi-\rm{i}\theta-2\xi_i} ;\beta) - (\tau+\tau_e)P(e^{\xi-\rm{i}\theta-2\xi_e};\beta)\quad & \textrm{in } \mathbb{R}^2 \setminus \overline{B_e}, \\[.5mm]
		(\tau+ \tau_i)P(e^{\xi-\rm{i}\theta-2\xi_i} ;\beta) + (\tau+\tau_e)P(e^{-\xi-\rm{i}\theta};\beta) \quad& \textrm{in } B_e \setminus\overline{ B_i}, \\[.5mm]
		-(\tau+ \tau_i)P(e^{-\xi-\rm{i}\theta} ;\beta) + (\tau+\tau_e)P(e^{-\xi-\rm{i}\theta};\beta) \quad& \textrm{in } B_i,
		\end{cases}\\[2mm]
		\frac{\partial q}{\partial \theta}& = \frac{\rm{i}}{2} \begin{cases}
		-(\tau+ \tau_i)P(e^{\xi-\rm{i}\theta-2\xi_i} ;\beta) + (\tau+\tau_e)P(e^{\xi-\rm{i}\theta-2\xi_e};\beta)\quad & \textrm{in } \mathbb{R}^2 \setminus \overline{B_e}, \\[.5mm]
		-(\tau+ \tau_i)P(e^{\xi-\rm{i}\theta-2\xi_i} ;\beta) + (\tau+\tau_e)P(e^{-\xi-\rm{i}\theta};\beta)\quad & \textrm{in } B_e \setminus \overline{B_i}, \\[.5mm]
		-(\tau+ \tau_i)P(e^{-\xi-\rm{i}\theta} ;\beta) + (\tau+\tau_e)P(e^{-\xi-\rm{i}\theta};\beta)\quad & \textrm{in } B_i.
		\end{cases}
		\end{align*}
		Thanks to Lemma \ref{lemma:LYu2D:series}, one can show that
		\beq\label{estimate:nablaU}
		\nabla U(\Bx) = r_* ^2 \nabla q(\Bx;\beta,\tau_i, \tau_e) + O(1), \quad \Bx \in \RR^2,
		\eeq
		in a similar way as in the proof of \cite[Lemma 5.1]{LYu2D}.
		Since the functions $x+\Re\{U(\bold{x})\}$ and $y+\Im\{U(\bold{x})\}$ are, respectively, the solutions to \eqnref{Eqn} with $H(\Bx)=x$ and $H(\Bx)=y$, we complete the proof of (a).
		
		Now we prove (b).
		From \eqnref{estimate:nablaU} it follows that
		\begin{align*}
		\nabla (u-H) (\bold{x})=
		& C_x r_*^2 (\Re\{\nabla q(\bold{x};\beta,\tau_i , \tau_e) \} \cdot \hat{\bold {e}}_\xi)\hat{\bold {e}}_\xi + C_x r_*^2 (\Re\{\nabla q(\bold{x};\beta,\tau_i , \tau_e) \} \cdot \hat{\bold {e}}_\theta)\hat{\bold {e}}_\theta \\
		+ &C_y r_*^2 (\Im\{\nabla q(\bold{x};\beta,\tau_i ,\tau_e) \} \cdot \hat{\bold {e}}_\xi)\hat{\bold {e}}_\xi + C_y r_*^2 (\Im\{\nabla q(\bold{x};\beta, \tau_i ,\tau_e) \} \cdot \hat{\bold {e}}_\theta)\hat{\bold {e}}_\theta+O(1).
		\end{align*}
		Using \eqnref{eqn:nabla} and \eqnref{Pproperty2}, we can show in a similar way as in \cite[section 4.1]{LYu2D} that 
		\begin{align}\label{eqn:nablaq_core_shell}
		\nabla q \cdot \hat{\bold {e}}_\xi & = {O}(1) + \frac{1}{2}h(\xi,\theta) \begin{cases}
		(\tau_i-\tau_e)P(e^{-\xi_e-\rm{i}\theta};\beta) & \textrm{in } \mathbb{R}^2 \setminus \overline{B_e}, \\[.5mm]
		(\tau_i+\tau_e+2\tau)P(e^{-\xi-\rm{i}\theta};\beta) & \textrm{in } B_e \setminus \overline{B_i}, \\[.5mm]
		(\tau_e-\tau_i)P(e^{-\xi-\rm{i}\theta};\beta) & \textrm{in } B_i,
		\end{cases}\\	\notag
		\nabla q \cdot \hat{\bold {e}}_\theta & = O(1) + \frac{\rm{i}}{2} h(\xi,\theta)(\tau_e-\tau_i) P(e^{-\xi-\rm{i}\theta};\beta)  \hspace{1.8cm} \textrm{in } \mathbb{R}^2.
		\end{align}

		By use of \eqnref{Pproperty1}, we obtain that
		$$\left|h(\xi,\theta) P(e^{-\xi-\rm{i}\theta};\beta)\right| \le \frac{1}{\alpha\beta}\le \frac{C}{|\ln \tau|}.$$
		Note that we have
		\begin{align*}
		|\tau_i|-|\tau_e | & \le \big| |\tau_i| - 1 \big|+ \big|1-|\tau_e|\big| \le -\ln	|\tau_i|- \ln|\tau_e| \le |\ln \tau|
		\end{align*}
		owing to $\tau> 0$ and $\left| \tau_i \right|, \left| \tau_e \right| < 1$.
		For $\tau_i, \tau_e<0$, it follows that
		$$ \left| \nabla q \right| \le C \quad \textrm{in } \mathbb{R}^2,$$
		so that the blow-up does not happen. For $\tau_i,\tau_e>0$, we have that
		\begin{align*}
		\left| \nabla q \cdot \hat{\bold {e}}_\xi\right|&\le C \textrm{ in }(\mathbb{R}^2 \setminus \overline{B_e})\cup B_i, \\
		\left| \nabla q \cdot \hat{\bold {e}}_\theta\right| &\le C \textrm{ in } \mathbb{R}^2.
		\end{align*}
		From \eqnref{eqn:nablaq_core_shell} and the definition of $h$, $\alpha$, we can derive the formula in (b) and complete the proof.
		 \end{proof}
	
	Figure \ref{fig:result} shows the graphs of the normal and tangential components of $\nabla u$ on $\p B_e$. As stated in Theorem \ref{blowupterm}(b), only the normal component blows up as $\ep$ tends to zero. 
	\begin{figure}[!h]
		\centering
		\begin{subfigure}{.32\textwidth}
			\centering
			\adjincludegraphics[height=4cm, width=5cm,trim={2.5cm 2.3cm 2cm 1.5cm}, clip]{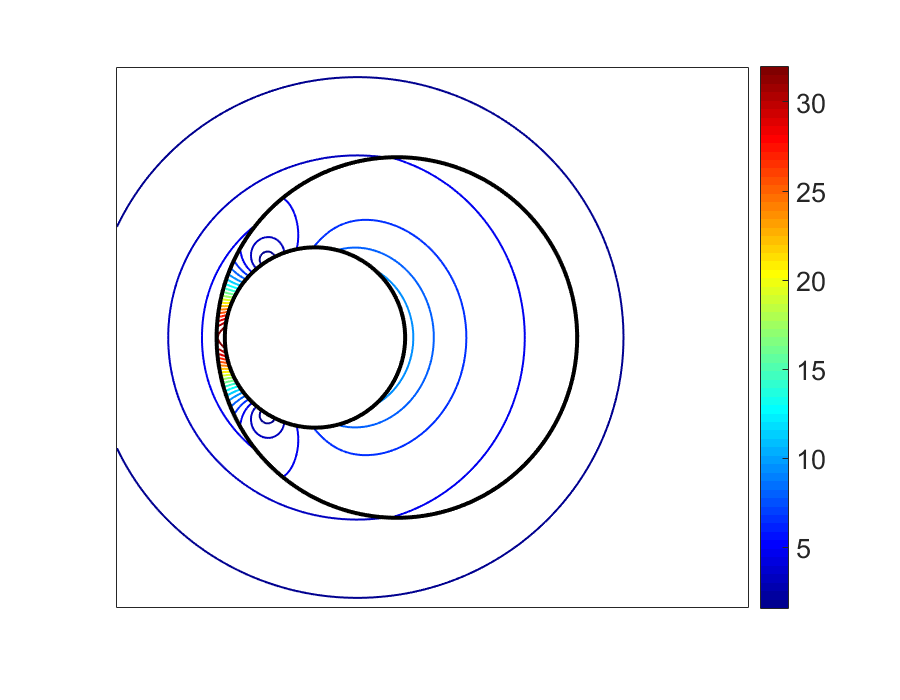}
			\caption{$\ep=0.1$, $H(\Bx)=x$}
		\end{subfigure}
		\begin{subfigure}{.32\textwidth}
			\centering
			\adjincludegraphics[height=4cm, width=5cm,trim={1.0cm 0.5cm 0.7cm 0.6cm}, clip]{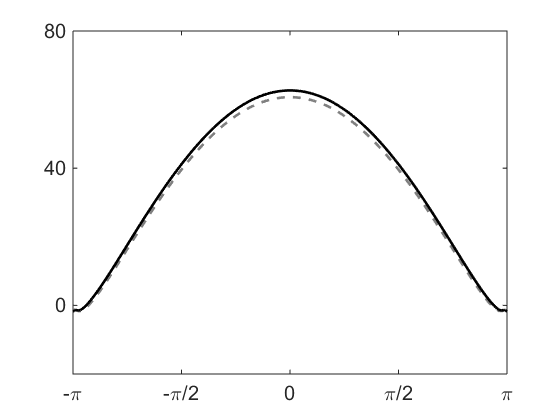}
			\caption{$\nabla (u-H) \cdot \hat{\bold {e}}_\xi$, $\ep=0.001$}
		\end{subfigure}
		\begin{subfigure}{.32\textwidth}
			\centering
			\adjincludegraphics[height=4cm, width=5cm,trim={1.0cm 0.5cm 0.7cm 0.6cm}, clip]{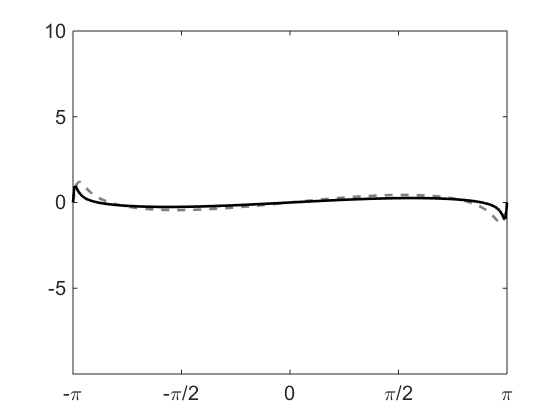}
			\caption{$\nabla (u-H) \cdot \hat{\bold {e}}_\theta$, $\ep=0.001$}
		\end{subfigure}
		\begin{subfigure}{.32\textwidth}
			\centering
			\adjincludegraphics[height=4cm, width=5cm,trim={2.5cm 2.3cm 2cm 1.5cm}, clip]{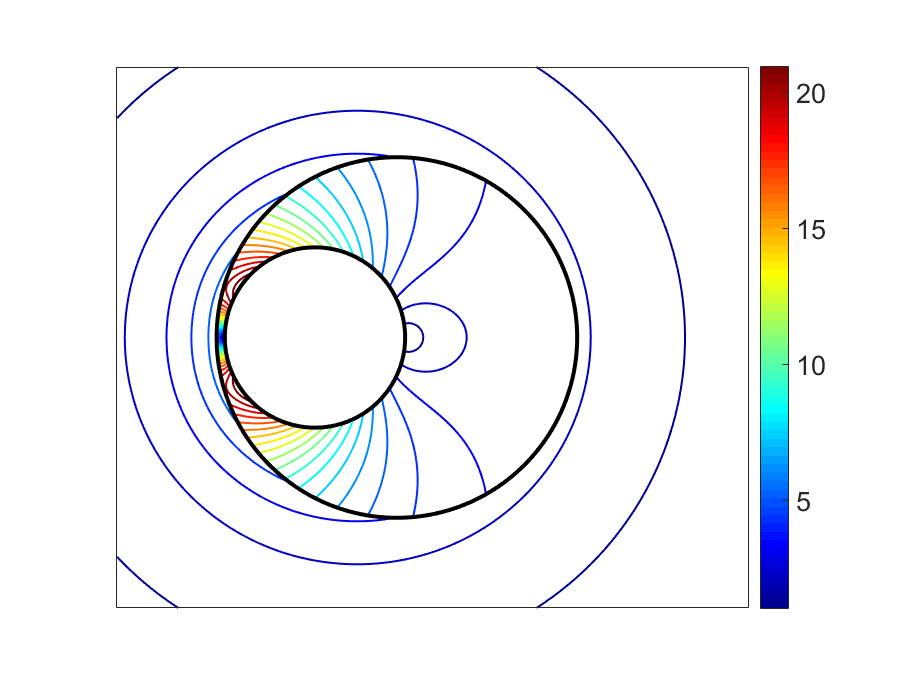}
			\caption{$\ep=0.1$, $H(\Bx)=y$}
		\end{subfigure}
		\begin{subfigure}{.32\textwidth}
			\centering
			\adjincludegraphics[height=4cm, width=5cm,trim={1.0cm 0.5cm 0.7cm 0.6cm}, clip]{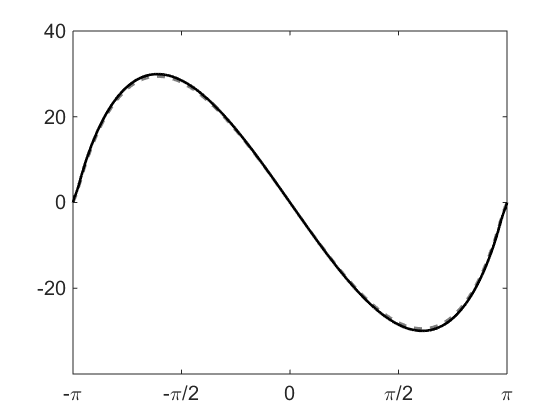}    \caption{$\nabla (u-H) \cdot \hat{\bold {e}}_\xi$, $\ep=0.001$}
		\end{subfigure} 
		\begin{subfigure}{.32\textwidth}
			\centering
			\adjincludegraphics[height=4cm, width=5cm,trim={1.0cm 0.5cm 0.7cm 0.6cm}, clip]{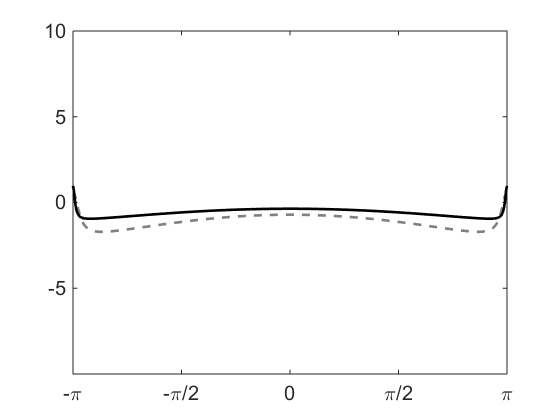}
			\caption{$\nabla (u-H) \cdot \hat{\bold {e}}_\theta$, $\ep=0.001$}
		\end{subfigure} 
			\caption{Electric field for the core-shell structure with $r_i=1$, $r_e=2$, $k_i=1$, and $k_e=0.01$,  where $\ep=0.1$ or $0.001$.
			The figures are contour plots of $|\nabla (u-H)|$ (the first column) and the graphs of $\nabla (u-H) \cdot \hat{\bold {e}}_\xi$ and $\nabla (u-H) \cdot \hat{\bold {e}}_\theta$ on $\p B_e$ (the second and third column, solid curve). Dashed curves are computed by using the singular term in Theorem \ref{blowupterm}(a). The background potential is $H(\Bx)=x$ (the first row) or $H(\Bx)=y$ (the second row).	}\label{fig:result}
	\end{figure}

	\subsection{Optimal bounds}
	Assume that $k_i= 1$ and $k_e \ll 1$. Then, by using \eqnref{P:integral}, one can easily deduce the upper and lower bounds for $\big| (\cosh\xi_i + \cos\theta) P(e^{-\xi_i - \rm{i}\theta};\beta) \big|$. 
	Specifically, we have that
	\begin{align}
	\bigg|\big(\cosh \xi_i + \cos\theta\big) \Re\left\{P(e^{-\xi_i - \rm{i}\theta};\beta) \right\} \bigg|
	&\geq \frac{C_1}{\beta+1}\quad \mbox{ for }\theta=0,\\
	\bigg|\big(\cosh \xi_i + \cos\theta\big) \Im\left\{P(e^{-\xi_i - \rm{i}\theta};\beta) \right\} \bigg|
	&\geq \frac{C_1}{(\beta+1)(\beta+3)}\quad \mbox{for }\theta=\frac{\pi}{2}
	\end{align}
	with some positive constant $C_1$ independent of $\ep,\beta$.
	Indeed, for $\theta = \frac{\pi}{2}$ we have
	\begin{align*}
	\bigg|\Im\left\{P(e^{-\xi_i - \frac{\pi}{2}{\rm i}};\beta) \right\}\bigg|
	&= \int_0 ^\infty e^{-\beta t} \frac{\sinh (\xi_i+t)}{2\cosh^2(\xi_i+t)} dt
	= e^{\beta \xi_i}\int_{\xi_i} ^\infty e^{-\beta t} \frac{e^t-e^{-t}}{(e^t + e^{-t})^2} dt\\
	&\ge e^{\beta \xi_i}\int_{\xi_i} ^\infty e^{-\beta t} \frac{e^t-e^{-t}}{(2e^{t})^2} dt\\
	&\ge \frac{1}{4} \left(\frac{e^{-\xi_i}}{\beta + 1} - \frac{e^{-3\xi_i}}{\beta + 3} \right)\\
	&\ge \frac{C_1}{(\beta+1)(\beta+3)}.
	\end{align*}
	
	Also, we can show that
	\begin{align}
	\Big| (\cosh\xi_i + \cos\theta) P(e^{-\xi_i - \rm{i}\theta};\beta) \Big| \le \frac{C_2}{\beta + 1}
	\end{align}
	for some positive constant $C_2$ independent of $\ep,\beta,\theta$. 
	Note that we have
	$$\frac{1}{\sqrt \ep (\beta + 1)} \approx \frac{1}{r_* k_e + \sqrt \ep} $$
	owing to
	$$\beta = \frac{r_*}{4\sqrt \ep}\big(1-\tau + O((1-\tau)^2)\big) \approx \frac{r_*}{\sqrt \ep} k_e.$$

	In total, assuming $k_i= 1$ and $k_e \ll 1$, the following holds with some positive constant $C_1,C_2$ independent of $k_e,k_i,\ep$:
	\begin{itemize}
		\item[$\bullet$] If $C_x \ne 0$ and $C_y =0$, then we have
		$$
		C_1 \frac{C_x}{k_e + \frac{\sqrt \ep}{r_*}} \le \Vert \nabla (u-H) \Vert_{L^\infty(B_e \setminus B_i)} \le C_2 \frac{C_x}{k_e + \frac{\sqrt \ep}{r_*}}.
		$$
		\item[$\bullet$] If $C_x = 0$, $C_y \ne 0$ and $\frac{r_* k_e}{\sqrt \ep} < C$ for some constant $C$ independent of $\ep$ ({\it i.e.}, $\beta$ is bounded), then we have
		$$
		C_1 \frac{C_y}{k_e + \frac{\sqrt \ep}{r_*}} \le \Vert \nabla (u-H) \Vert_{L^\infty(B_e \setminus B_i)} \le C_2 \frac{C_y}{k_e + \frac{\sqrt \ep}{r_*}}.
		$$
	\end{itemize}

	\section{Image line charge formula}\label{sec:cartesian}
	We obtained an asymptotic formula for the electric field in bipolar coordinates in the previous section, where the singular term is a linear combination of the Lerch transcendent functions. 
	In this section, we rewrite the formula in terms of the Cartesian coordinates. It turns out that the singular term is actually the single and double layer potentials with line charges on the $x$-axis. Also, we derive an asymptotic formula for the case of two separated disks. As well as the core-shell geometry, the singular term for two separated disks can be expressed as the single and double layer potentials with line charges on the $x$-axis.

	\subsection{Line charge formula of the Lerch transcendent function}\label{sec:lerch}
	
	We first express each component of the singular function $q$, that is, a linear combination of the Lerch transcendent function, in the Cartesian coordinates. 
	Before stating the formulas, let us define some notations. Here and after, we denote by $[ \Bx_1, \Bx_2 ]$ the line segment connecting the two points $\Bx_1, \Bx_2$. We also write $f(s,0)=f(s)$ for a function $f$ defined on $\RR$. 

	\begin{lemma} \label{Lerch:xy}
		Assume $\xi_0\geq0$. We denote by $\bold{c_0}=(c_0,0)$ and $r_0$ the center and radius of the level curve $\{\xi=\xi_0\}$ for $\xi_0\neq0$. We denote by $\bold{c_0}=(+\infty,0)$ and $r_0=\infty$ for $\xi_0=0$. We set the charge distribution functions as
		\begin{align*}
		\ds\varphi_+(s)&= 2\alpha \beta e^{2\beta\xi_0}\frac{(s-\alpha)^{\beta-1}}{(s+\alpha)^{\beta+1}}
		\ds,\quad \psi_+(s)=e^{2\beta\xi_0} \left(\frac{s-\alpha}{s+\alpha}\right)^\beta
		\quad \mbox { for } \alpha<s<c_0,\\
		\ds\varphi_-(s)& = 2\alpha \beta e^{2\beta\xi_0} \frac{(s+\alpha)^{\beta-1}}{(s-\alpha)^{\beta+1}},
		\ds\quad\psi_-(s) = -e^{2\beta\xi_0}\left(\frac{s+\alpha}{s-\alpha}\right)^\beta
		\quad \mbox { for } -c_0<s<-\alpha.
		\end{align*}
		Then, the following holds:
		\begin{itemize}
			\item[\rm(a)] For $\xi<2\xi_0$, we have
			$$L(e^{\xi +\rm{i}\theta - 2\xi_0};\beta)= \int_{[\bold{p}_2,\,\bold{c_0}]} \ln|\Bx-\Bs| \varphi_+(\Bs)\,d\sigma(\Bs) + r_+(\Bx)
			-\rm{i}\int_{[\bold{p}_2,\,\bold{c_0}]} \frac{ \langle \Bx-\Bs, \,\Be_2 \rangle}{|\Bx-\Bs|^2} \psi_+(\Bs)\,d\sigma(\Bs),
			$$
			where the remainder term satisfies that $|\nabla r_+(\Bx) |\leq\frac{1}{r_0}$ for all $\Bx$. 
			\item[\rm(b)] For $\xi>-2\xi_0$, we have
			$$L(e^{-\xi -\rm{i}\theta - 2\xi_0};\beta)=\int_{[-\bold{c_0}, \,\bold{p_1}]} \ln|\Bx-\Bs| \varphi_-(\Bs)\,d\sigma(\Bs) + r_-(\Bx)-\rm{i}\int_{[-\bold{c_0},\, \bold{p_1}]} \frac{ \langle \Bx-\Bs,\, \Be_2\rangle}{|\Bx-\Bs|^2} \psi_-(\Bs)\,d\sigma(\Bs),$$
			where the remainder term satisfies 
			that $|\nabla r_-(\Bx) |\leq\frac{1}{r_0}$ for all $\Bx$. 
			
		\end{itemize}
		
	\end{lemma}

	\begin{proof}
		By integration by parts we have
		\begin{align*}
		L(z;\beta)&  = -\int_0 ^\infty \frac{ze^{-t}e^{-\beta t} }{1+z e^{-t}}\,dt \\
		& = -\ln (1+z) + \beta \int_0 ^{\infty} \ln(1+ze^{-t}) e^{-\beta t}\, dt \\
		&= \beta \int_0 ^{\infty} \ln \left(\frac{1+ze^{-t}}{1+z}\right) e^{-\beta t} \,dt.
		\end{align*}
		Here, we used the fact that $\beta>0$ to evaluate the integrand at $t=\infty$.
		Owing to $\overline{L(z;\beta)} = L(\overline{z};\beta)$, it follows that
		\begin{align*}
		\Re \left\{L(z;\beta)\right\} &= \frac{1}{2} \left(L(z;\beta)+L(\overline{z};\beta) \right) = \beta \int_0 ^{\infty} \ln \left|\frac{1+ze^{-t}}{1+z}\right| e^{-\beta t} \,dt.
		\end{align*}	
		The imaginary part of $L(z;\beta)$ becomes
		\begin{align}\notag
		\ds\Im \big\{L(z;\beta)\big\} &= \frac{1}{2\rm{i}} \left(L(z;\beta)-L(\overline{z};\beta) \right)\notag \\\notag
		&= -\frac{1}{2\rm{i}} \int_0 ^\infty \frac{ze^{-(\beta+1)t}}{1+z e^{-t}} - \frac{\overline{z} e^{-(\beta+1)t}}{1+\overline{z} e^{-t}} \,dt\\
		\ds&= -\frac{1}{2\rm{i}} \int_0 ^\infty \frac{(z-\overline{z}) e^{-(\beta+1)t}}{|1+z e^{-t}|^2}\,dt
		= - \int_0 ^\infty \frac{\Im\{z\} e^{-(\beta+1)t}}{|1+z e^{-t}|^2}\,dt.\label{linecharge:imaginary}
		\end{align}

		To prove (a), we now assume $\xi<2\xi_0$ and set $z=e^{\xi +\rm{i}\theta - 2\xi_0}$.
		We first consider $\Re \left\{L(z;\beta)\right\}$. 
		Note that $|z| \le 1$. We compute
		\begin{align} 
		1+ze^{-t} & = 1+e^{\xi+\rm{i}\theta} e^{-(2\xi_0 +t)}\notag \\
		&= (1-e^{-(2\xi_0+t)})(1+e^{\xi+\rm{i}\theta})\left( \frac{1}{1+e^{\xi+\rm{i}\theta}} + \frac{e^{-(2\xi_0+t)}}{1-e^{-(2\xi_0+t)}} \right) \notag \\ 
		& = (1-e^{-(2\xi_0+t)})\left(\frac{2\alpha}{\alpha-z(\xi,\theta)}\right) \left(\frac{\alpha - z(\xi,\theta)}{2\alpha} + \frac{e^{-(2\xi_0+t)}}{1-e^{-(2\xi_0+t)}} \right) \notag \\ \label{coordchange}
		& = \frac{1-e^{-(2\xi_0+t)}}{\alpha - z(\xi,\theta)}\left(-z(\xi,\theta) + \alpha \frac{1+e^{-(2\xi_0+t)}}{1-e^{-(2\xi_0+t)}} \right).
		\end{align}
		For ease of estimation, we set
		$$
		s(t)=\alpha \frac{1+e^{-(2\xi_0+t)}}{1-e^{-(2\xi_0+t)}}.$$
		In other words, $$e^{-t} = e^{2\xi_0}\frac{s(t)-\alpha}{s(t)+\alpha}.
		$$
		From \eqnref{coordchange} it follows that 
		$$1+ze^{-t}=\frac{1-e^{-(2\xi_0+t)}}{\alpha - z(\xi,\theta)}(-z(\xi,\theta)+s(t)).$$
		It is immediately clear that $s(t)$ is the center of circle $\{\xi=\xi_0+\frac{t}{2}\}$ owing to \eqnref{bipolar:circle}. In particular, $s(0)=c_0$. Hence we have for $\xi_0 > 0$
		\begin{align*}
		&\Re \left\{L(e^{\xi +\rm{i}\theta - 2\xi_0};\beta)\right\} \\
		&=\beta \int_0 ^{\infty} \ln \left|\frac{1+ze^{-t}}{1+ze^{-0}}\right| e^{-\beta t}\, dt\\
		&= \beta \int_0 ^{\infty} \ln \left| \frac{(1-e^{-(2\xi_0+t)})\left(z(\xi,\theta) - s(t) \right)}{\left(1-e^{-2\xi_0} \right) \left( z(\xi,\theta)-s(0)\right)}\right| e^{-\beta t}\,dt\\
		&=\beta \int_0 ^{\infty} \ln\big|z(\xi,\theta)-s(t)\big| e^{-\beta t} \,dt
		-\beta \int_0 ^{\infty} \ln\big|z(\xi,\theta)-s(0)\big|e^{-\beta t} \,dt
		+ L(-e^{-2\xi_0};\beta).
		\end{align*}
		Due to the fact that
		$$
		ds =\frac{-2\alpha e^{-(2\xi_0+t)}}{\big(1-e^{-(2\xi_0+t)}\big)^2}dt = -\frac{1}{2\alpha}(s+\alpha)(s-\alpha)dt,
		$$
		we have
		\begin{align*}
		\Re \left\{L(e^{\xi +\rm{i}\theta - 2\xi_0};\beta)\right\} 
		&= \beta \int_0 ^{\infty} \ln\big|z(\xi,\theta)-s(t)\big| e^{-\beta t}\, dt - \ln|z(\xi,\theta)-c_0|+L(-e^{-2\xi_0};\beta)\\
		&= 2\alpha \beta e^{2\beta \xi_0} \int_{\alpha} ^{c_0} \ln\big|z(\xi,\theta)-s\big| \frac{(s-\alpha)^{\beta-1}}{(s+\alpha)^{\beta+1}} \,ds + r_+(z(\xi,\theta))
		\end{align*}
		with $$r_+(z(\xi,\theta)) = - \ln\big|z(\xi,\theta)-c_0\big|+L(-e^{-2\xi_0};\beta).$$ One can easily see that 
		$$\big|\nabla r(z(\xi,\theta))\big|\leq\frac{1}{r_0}\quad\mbox{for all }\xi<2\xi_0.$$
		Similar estimates hold for $\xi_0 = 0$.
		
		We now consider $\Im \left\{L(z;\beta)\right\}$. 
		Note that
		\begin{align*}
		\Im \big\{e^{\xi+\rm{i}\theta}\big\} &= \Im \left\{ \frac{\alpha+z(\xi,\theta)}{\alpha - z(\xi,\theta)}\right\}\\
		& = \Im \left\{ \frac{(\alpha+z(\xi,\theta))(\alpha-\overline{z(\xi,\theta)})}{\big|\alpha-z(\xi,\theta)\big|^2}\right\}\\
		&=\frac{2\alpha}{\big|\alpha-z(\xi,\theta)\big|^2} \Im \big\{z(\xi,\theta)\big\}.
		\end{align*}
		Using \eqnref{linecharge:imaginary} and \eqnref{coordchange}, we have
		\begin{align*}
		\Im \left\{L(e^{\xi +\rm{i}\theta - 2\xi_0};\beta)\right\} &=- \int_0 ^\infty \frac{\Im\{e^{\xi +\rm{i}\theta - 2\xi_0}\} e^{-(\beta+1)t}}{\big|1+e^{\xi +\rm{i}\theta - 2\xi_0-t}\big|^2}\,dt\\[2mm]
		&= \int_0 ^\infty \frac{\Im\big\{z(\xi,\theta)\big\}e^{-\beta t} }{\big|z(\xi,\theta) - s(t) \big|^2}\frac{-2\alpha e^{-(2\xi_0+t)}}{\big(1-e^{-(2\xi_0+t)}\big)^2} \,dt\\[2mm]
		&= -e^{2\beta \xi_0} \int_{\alpha} ^{c_0} \frac{\Im \big\{z(\xi,\theta)\big\}}{\big|z(\xi,\theta)-s\big|^2}\left(\frac{s-\alpha}{s+\alpha}\right)^\beta\, ds.
		\end{align*}
		Note that $\Im \big\{ z(\xi,\theta) \big\} = \langle \Bx-\Bs, \Be_2\rangle$.
		Hence we complete the proof of (a).
		
		We can prove (b) in a similar way, where the remainder term is $r_-(\Bx) = - \ln\big|z(\xi,\theta)+c_0\big|+L(-e^{-2\xi_0};\beta)$. 
		 \end{proof}

	\subsection{Asymptotic formula for the core-shell geometry}

	We now express the singular part of $u$ in terms of the single and double layer potential with image line charges on the $x$-axis. 
	
	\begin{theorem} \label{asymptotic}
		Let $u$ be the solution to \eqnref{Eqn} corresponding to an entire harmonic function $H$. 
		For any $0<k_i,k_e\neq 1<\infty$, we have that $\|\nabla (u-H)\|_{L^\infty(\RR^2\setminus\overline{(B_e\setminus B_i)})}\leq C$ for some constant $C$ independent of $\ep$.
		The gradient blow-up for $u$ may occur only when $\tau_i,\tau_e$ are similar to $1$ (that is, equivalently, $0<k_e\ll 1,k_i$). 
		
		For $k_i, k_e$ such that $0<k_e<1$ and $k_e<k_i$, we have that in $B_e\setminus B_i$,
		\begin{align*}
		u(\Bx)&=  H(\Bx)+ \frac{C_x r_*^2}{2} \left[ \int_{((-\infty,0),\, \bold{p_1}]} \ln|\Bx-\Bs|\varphi_1(\Bs) \,d\sigma(\Bs) -\int_{[\bold{p}_2,\,\bold{c_i}]} \ln|\Bx-\Bs| \varphi_2(\Bs)\, d\sigma(\Bs) \right] \\ 
		& + \frac{C_yr_*^2}{2}\left[ \int_{((-\infty,0), \,\bold{p_1}]} \frac{ \langle \Bx-\Bs, \Be_2\rangle}{|\Bx-\Bs|^2} \psi_1(\Bs)\, d\sigma(\Bs)  -\int_{[\bold{p}_2,\,\bold{c_i}]} \frac{ \langle \Bx-\Bs, \Be_2 \rangle}{|\Bx-\Bs|^2} \psi_2(\Bs)\, d\sigma(\Bs)\right]  + r(\Bx),
		\end{align*}
		where the charge distributions are
		\begin{align*}
		\varphi_1(s) &= \left(\tau +\tau_e \right)2\alpha\beta \, \frac{(s+\alpha)^{\beta-1}}{(s-\alpha)^{\beta+1}}\,\chi_{(-\infty,-\alpha]}(s),\\
		\psi_1(s) &= \left(\tau+\tau_e\right) \left(\frac{s+\alpha}{s-\alpha}\right)^\beta\chi_{(-\infty,-\alpha]}(s),\\
		\varphi_2(s) &=   \left(\tau + \tau_i\right)2\alpha \beta e^{2\beta \xi_i} \, \frac{(s-\alpha)^{\beta-1}}{(s+\alpha)^{\beta+1}}\,\chi_{[\alpha,\,c_i]}(s),\\
		\psi_2(s) &= \left(\tau+\tau_i\right) e^{2\beta \xi_i}\left(\frac{s-\alpha}{s+\alpha}\right)^\beta\chi_{[\alpha,\,c_i]}(s),
		\end{align*}
		and $|\nabla r(\Bx)|$ is uniformly bounded independently of $\ep$ and $\Bx$.
	\end{theorem}
	\begin{proof}
		Applying Lemma \ref{Lerch:xy} to Theorem \ref{blowupterm} (a) we prove the theorem.
		 \end{proof}

	\begin{figure}[!]
		\centering
		\includegraphics[width=7cm,height=4.2cm,trim={2cm 3cm 2cm 2.5cm}, clip]{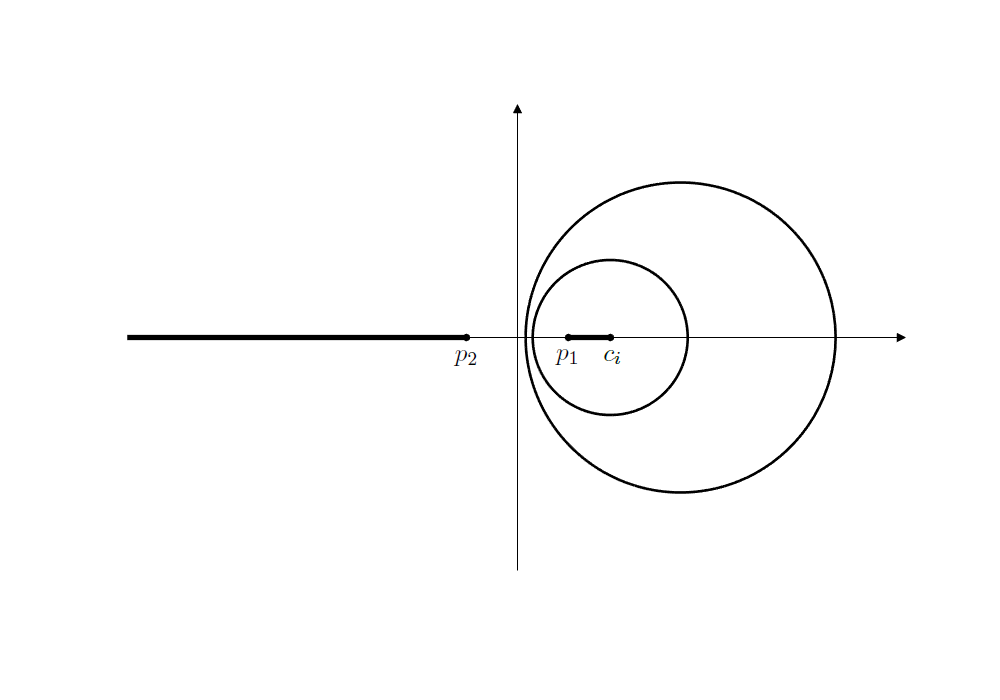}
		\caption{Bold intervals indicate the support of line charges for the core-shell structure.}
		\label{fig:support:core_shell}
		\vskip .2cm
		\begin{subfigure}{1\textwidth}
			\centering
			\adjincludegraphics[width=6cm, height=4.5cm, trim={.8cm .5cm .5cm .4cm}, clip]{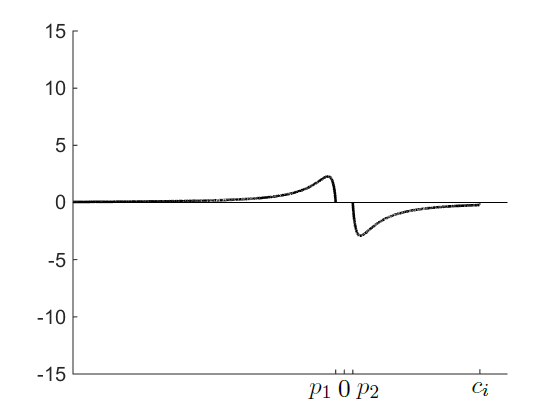}
			\hskip 1cm
			\adjincludegraphics[width=6cm, height=4.5cm, trim={.8cm .5cm .5cm .4cm}, clip]{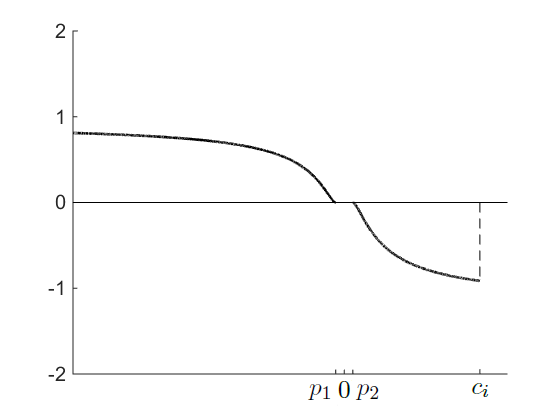}
			\caption{ $k_i = 1,\, k_e=0.03,\, \ep = 0.001$ ($\beta = 1.8979, \tau = 0.8869$)}
		\end{subfigure}	
		\begin{subfigure}{1\textwidth}
			\centering
			\adjincludegraphics[width=6cm,height=4.5cm,trim={.8cm .5cm .5cm .4cm}, clip]{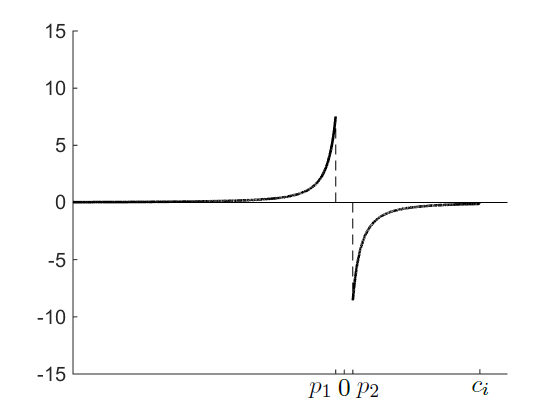}
			\hskip 1cm
			\adjincludegraphics[width=6cm,height=4.5cm,trim={.8cm .5cm .5cm .4cm}, clip]{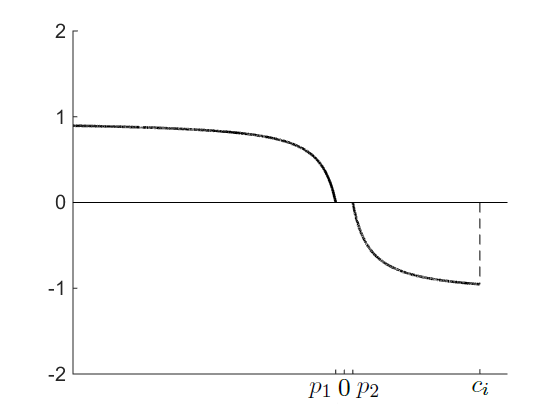}
			\caption{$k_i = 1,\, k_e= 0.0158,\, \ep = 0.001$ ($\beta=1, \tau=0.9387$)}
		\end{subfigure}	
		\begin{subfigure}{1\textwidth}
			\centering
			\adjincludegraphics[width=6cm,height=4.5cm,trim={.8cm .5cm .5cm .4cm}, clip]{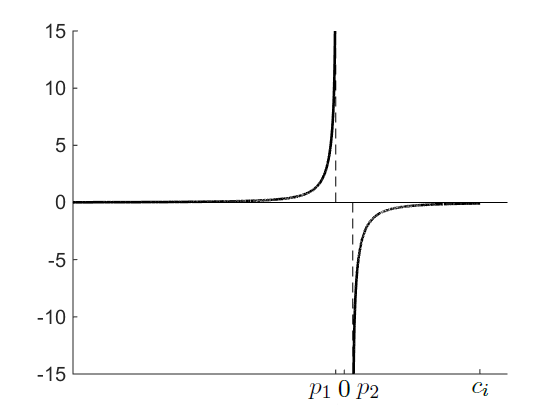}
			\hskip 1cm
			\adjincludegraphics[width=6cm,height=4.5cm,trim={.8cm .3cm .5cm .4cm}, clip]{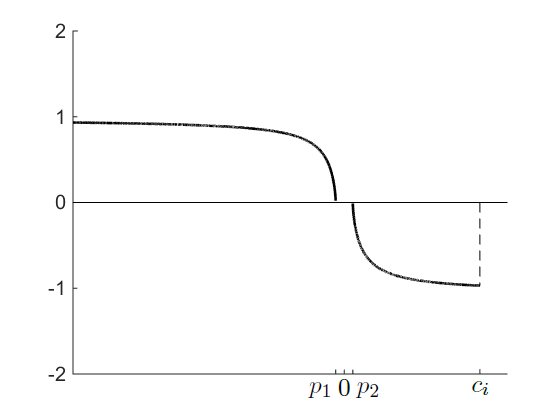}
			\caption{$k_i = 1,\, k_e=0.01,\, \ep = 0.001$ ($\beta = 0.6325, \tau = 0.9608$)}
		\end{subfigure}	
\caption{Image line charges for the core-shell structure with $r_i=1,\, r_e=2,\, k_i=1,\, \ep=0.001$ and various $k_e$-values. The left column shows $\frac{1}{2}(\varphi_1-\varphi_2)$, and the right column shows $\frac{1}{2}(\psi_1-\psi_2)$ in Theorem \ref{asymptotic}. 
	As $k_e$ decreases, $\beta$ approaches to $0$. The density $\frac{1}{2}(\varphi_1-\varphi_2)$ is bounded if $\beta\leq 1$, but diverges to $\infty$ near $p_1$ and $p_2$ if $\beta<1$.}\label{fig:density:core_shell}
	\end{figure}

\begin{figure}[!]
	\begin{center}
		\begin{subfigure}{1\textwidth}
			\centering
			\adjincludegraphics[width=6cm, height=4.5cm, trim={.8cm .3cm .5cm .4cm}, clip]{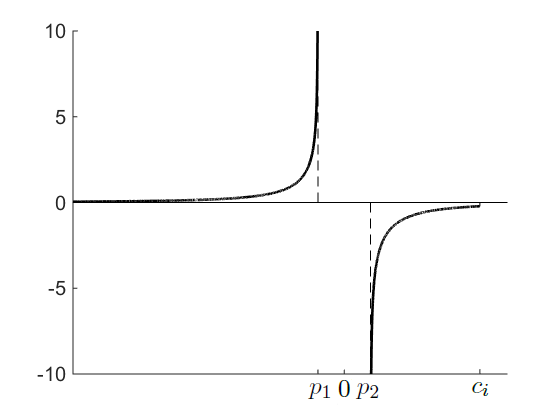}
			\hskip 1cm
			\adjincludegraphics[width=6cm, height=4.5cm,trim={.8cm .3cm .5cm .4cm}, clip]{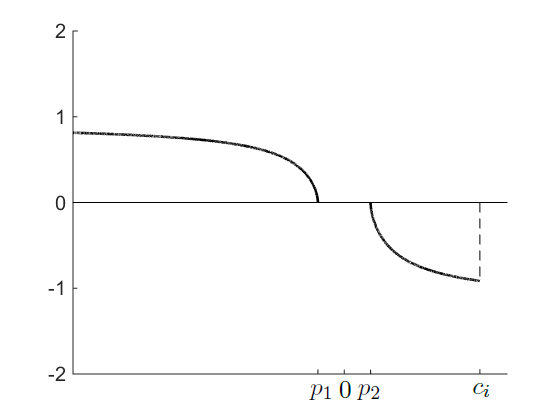}
			\caption{$k_i = 1,\, k_e = 0.03,\, \ep=0.01$ ($\beta = 0.6002, \tau = 0.8869$)}
		\end{subfigure} 
		\begin{subfigure}{1\textwidth}
			\centering
			\adjincludegraphics[width=6cm,height=4.5cm,trim={.8cm .5cm .5cm .4cm}, clip]{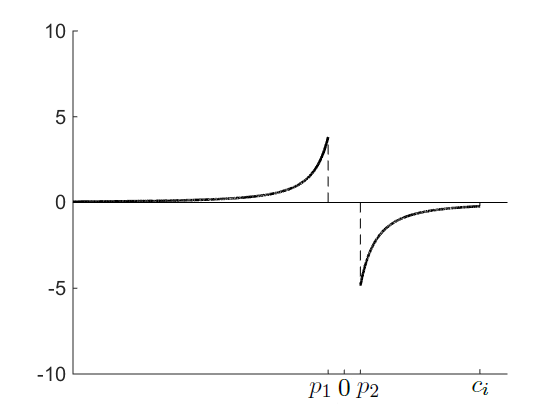}
			\hskip 1cm
			\adjincludegraphics[width=6cm,height=4.5cm,trim={.8cm .5cm .5cm .4cm}, clip]{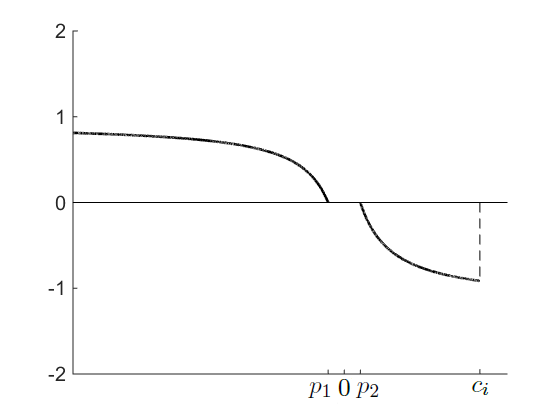}
			\caption{$k_i = 1,\, k_e = 0.03,\,\ep=0.0036$ ($\beta=1, \tau = 0.8869$)}
		\end{subfigure} 
		\begin{subfigure}{1\textwidth}
			\centering
			\adjincludegraphics[width=6cm,height=4.5cm, trim={.8cm .5cm .5cm .4cm}, clip]{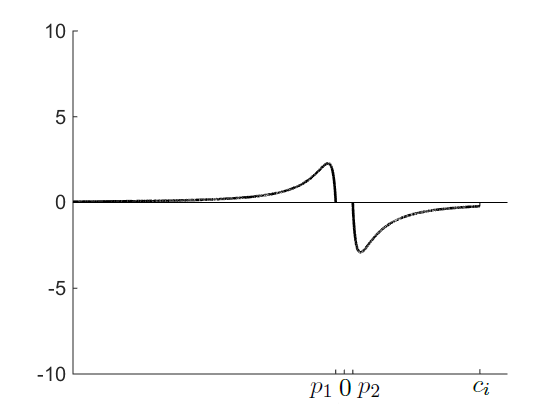}
			\hskip 1cm
			\adjincludegraphics[width=6cm,height=4.5cm, trim={.8cm .5cm .5cm .4cm}, clip]{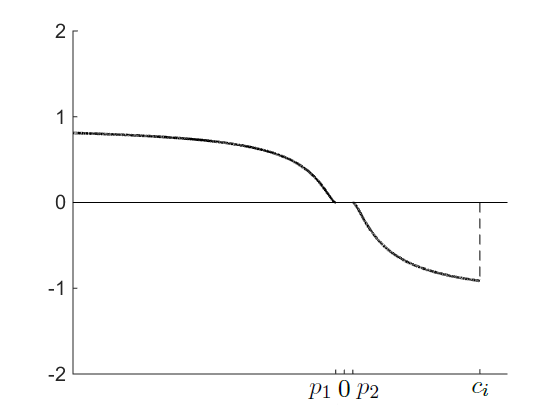}
			\caption{$k_i = 1,\, k_e = 0.03,\,\ep=0.001$ ($\beta = 1.8979, \tau = 0.8869$)}
		\end{subfigure} 
		
		\caption{Image line charges for the core-shell structure with $r_i=1,\, r_e=2,\, k_i=1,\,k_e=0.03$ and various $\ep$-values. The left column shows $\frac{1}{2}(\varphi_1-\varphi_2)$, and the right column shows $\frac{1}{2}(\psi_1-\psi_2)$ in Theorem \ref{asymptotic}. The support of line charges for the core-shell structure is drawn in Figure \ref{fig:support:core_shell}. As $\ep$ decreases, $\beta$ increases so that $\frac{1}{2}(\varphi_1-\varphi_2)$ becomes bounded.}\label{fig:density:core_shell_2}
	\end{center}
\end{figure}
	
	Figure \ref{fig:support:core_shell} shows the support of the density functions in Theorem \ref{asymptotic}. 
	One can easily see that  $\frac{\p \psi_j (s)}{\p s} = \varphi_j (s)$ for $j=1,2$ and from \eqnref{bipolar:circle},
	\begin{gather*}
	\int \varphi_1(s)\,ds=\tau+\tau_e,\quad \int\varphi_2(s)\,ds=\tau+\tau_i.
	\end{gather*}
	The density functions show a different feature depending on $\beta$. The functions $\varphi_1$ and $\varphi_2$ are bounded if $\beta\leq 1$, but they diverge to $\pm\infty$ near $p_1$ or $p_2$ if $\beta<1$. See Figure \ref{fig:density:core_shell} and Figure \ref{fig:density:core_shell_2}.

	

	
	\begin{remark}
		When $\beta=0$ (equivalently $\tau_i=\tau_e=1$), which is the limiting case of $|\ln \tau|\ll \sqrt{\ep}$, we have that \begin{align*}
		L(e^{\xi+\rm{i}\theta - 2\xi_0}; 0) &= \ln|\Bx-\Bp_2| + \widetilde{r}_+ (\Bx) - {\rm i}\int_{[\bold{p}_2,\,\bold{c_0}]} \frac{ \langle \Bx-\Bs, \Be_2 \rangle}{|\Bx-\Bs|^2}\, d\Bs,\quad \xi<2\xi_0,\\
		L(e^{-\xi-\rm{i}\theta - 2\xi_0}; 0) &= \ln|\Bx-\Bp_1| + \widetilde{r}_- (\Bx) + {\rm i}\int_{[-\bold{c_0}, \,\bold{p_1}]} \frac{ \langle \Bx-\Bs, \Be_2\rangle}{|\Bx-\Bs|^2}\, d\Bs, \quad \xi>-2\xi_0,
		\end{align*}
		where $|\nabla \widetilde{r}_+(\Bx)|$ and $|\nabla \widetilde{r}_-(\Bx)|$ are bounded independently of $\ep,\Bx$.
		Hence, the singular term of $u$ in \eqnref{u:singular} satisfies
		\begin{align}
		&C_x r_*^2\Re\big\{q\big(\Bx;0,1, 1\big)\big\} + C_y r_*^2\Im\big\{q\big(\Bx;0,1, 1\big)\big\}\notag \\[1mm]\label{beta_zero}
		=\ &  C_x r_*^2  \ln\left|\frac{\Bx-\Bp_1}{\Bx-\Bp_2}\right| + C_y r_*^2 \int_{[\bold{p}_2,\,\bold{c_0}] \cup [-\bold{c_0}, \,\bold{p_1}]} \frac{ \langle \Bx-\Bs, \, \Be_2 \rangle}{|\Bx-\Bs|^2} \,d\Bs +\widetilde{r}(\Bx),
		\quad \Bx\in B_e \setminus B_i,
		\end{align}
		where $|\nabla \widetilde{r}(\Bx)|$ are bounded independently of $\ep,\Bx$.
		\smallskip
		
		Recall that we denote by $\Bx_i=(x_i,0)$ and $\Bx_e=(x_e,0)$ the closest points on $\p  B_i$ and $\p B_e$ between $\p B_i$ and $\p B_e$, respectively. The gradient of the function in \eqnref{beta_zero} blows up at $\Bx_e$ as $\epsilon$ tends zero, since the distance between $\Bp_1$ and $\Bp_2$ is $O(\sqrt{\epsilon})$. 
		However, the solution to the following problem does not blow up:
		\begin{equation} \label{Eqn_insul}
		\left\{ \begin{array}{ll}
		\ds \Delta u_0 = 0 \qquad& \textrm{in } \mathbb{R}^2\setminus \overline{B_e}, \\
		\ds \pd{u_0}{\nu}=0\quad&\mbox{on }\p B_e,\\
		\ds u_0(\bold{x}) - H(\bold{x}) = O(|\bold{x}|^{-1}) \qquad& \textrm{as } |\bold{x}| \to \infty.
		\end{array} \right.
		\end{equation}
		These facts imply that $\nabla u(\Bx_e)$ for the solution $u$ to \eqnref{Eqn} does not converge to $\nabla u_0(\Bx_e)$ when $k_e\rightarrow\infty$ and $\ep\rightarrow 0$, satisfying that $k_i$ is fixed and $|\ln \tau|\ll \sqrt{\ep}$. 
		See \cite[section 5.3]{LYu2D} for a non-convergence result for $\nabla u$ in the case of separated disks. 
	\end{remark}
	
	
	\subsection{Asymptotic formula for the case of two separated disks}
	We now consider the gradient blow-up feature for the solution $u$ to \eqnref{Eqn} in the presence of two separated disks, say $B_1$ and $B_2$. In other words, the conductivity profile is given by
	\beq\label{sigma:separated}
	\sigma = k_1 \chi(B_1) + k_2 \chi(B_2) + \chi \big(\mathbb{R}^2 \setminus (B_1 \cup B_2)\big).
	\eeq
	As before we assume that $0<k_i,k_e\neq 1<\infty$.
	{For two separated disks with arbitrary conductivities, the asymptotic formula for the potential function in bipolar coordinates was derived in \cite{LYu2D}.
		In this section, we express the singular term by the single and double layer potential with line charges.}

	We set $B_j$ to be the disk centered at $(c_j,0)$ with radius $r_j$, $j=1,2$. We denote by $\ep$ the distance between $B_1$ and $B_2$; see Figure \ref{fig:support:separated} for the geometry of the two separated disks. We set $R_j$ to be the reflection with respect to $\p B_j$, $j=1,2$. We can assume that, by appropriate shifting and rotation, the fixed points of the combined reflection $R_1 \circ R_2$ are $\Bp_1=(-\alpha,0),~ \Bp_2 = (\alpha,0)$ with
	$$
	\alpha = \frac{\sqrt{\epsilon(2r_1+\epsilon)(2r_2+\epsilon)(2r_1+2r_2+\epsilon)}}{2(r_1+r_2+\epsilon)}.
	$$
	As for the core-shell geometry,  we denote by $(\xi,\theta)$ the bipolar coordinates defined with respect to the poles $\Bp_1$ and $\Bp_2$.
	
	We now set $$r_* = \sqrt{\frac{2r_1r_2}{r_1+r_2}}$$  and 
	$$\tau_1=\frac{k_1-1}{k_1+1}, \quad \tau_2=\frac{k_2-1}{k_2+1}, \quad \tau=\tau_1\tau_2, \quad \beta = \frac{r_*(-\ln \tau)}{4\sqrt{\epsilon}}.$$

	\begin{lemma}[\cite{LYu2D}]
		Let $u$ be the solution to \eqnref{Eqn} with $\sigma$ given by \eqnref{sigma:separated}.
		Then $u$ satisfies that
		\beq\label{u:singular_disjoint}
		u(\Bx) = H(\Bx) + C_x r_*^2\Re\big\{\widetilde{q}\big(\Bx;\beta,\tau_i, \tau_e\big)\big\} + C_y r_*^2\Im\big\{\widetilde{q}\big(\Bx;\beta,-\tau_i, -\tau_e\big)\big\} + r(\Bx),\quad \Bx\in\RR^2,
		\eeq
		where 
		\begin{align*}
		&\widetilde{q}(\Bx;\beta,\tau_1,\tau_2) \\
		&:= \frac{1}{2}
		\begin{cases}
		\ds(\tau + \tau_1) L(e^{-(\xi+\rm{i}\theta)-2\xi_1};\beta) - (\tau + \tau_2) L(e^{(\xi+\rm{i}\theta)-2\xi_2};\beta) & \mbox{in } \RR^2 \setminus \overline{(B_1 \cup B_2)}, \\[.5mm]
		\ds(\tau + \tau_1) L(e^{\xi-\rm{i}\theta};\beta) - (\tau + \tau_2) L(e^{(\xi+\rm{i}\theta)-2\xi_2};\beta) & \mbox{in } B_1 ,\\[.5mm]
		\ds(\tau + \tau_1) L(e^{-(\xi+\rm{i}\theta)-2\xi_1};\beta) - (\tau + \tau_2) L(e^{-\xi+\rm{i}\theta};\beta) & \mbox{in } B_2,
		\end{cases}
		\end{align*}
		and $\Vert \nabla r(\Bx) \Vert_{L^\infty(\RR^2)} < C$ for some constant $C$ independent of $\epsilon$. 
	\end{lemma}

	Using the line charge formula of the Lerch transcendent function in section \ref{sec:lerch}, one can easily show the following.
	\begin{theorem}\label{thm:imagecharge:separate} Let $u$ be the solution to \eqnref{Eqn} with $\sigma$ as given by \eqnref{sigma:separated}.
		The gradient blow-up for $u$ may occur only when $\tau$ is similar to $1$ (equivalently, $k_1,k_2\gg 1$ or $0<k_1,k_1\ll 1$). 
		Moreover, $u$ satisfies the following integral formulas in $\RR^2 \setminus (B_1 \cup B_2)$:
		\begin{itemize}
			\item[\rm(a)] For $1<k_1, k_2 < \infty$, we have 
			\begin{align*}
			u(\Bx) = H(\Bx) + \frac{C_x r_*^2}{2} \left[ \int_{[\Bc_1, \Bp_1]} \ln|\Bx-\Bs| \varphi_1(\Bs)\, d\sigma(\Bs) -\int_{[\bold{p}_2,\Bc_2]} \ln|\Bx-\Bs| \varphi_2(\Bs)\, d\sigma(\Bs) \right] + r(\Bx),
			\end{align*}
			where the charge distributions are
			\begin{align*}
			\varphi_1 (s) & = (\tau + \tau_1)  2\alpha \beta e^{2\beta\xi_1} \, \frac{(s+\alpha)^{\beta-1}}{(s-\alpha)^{\beta+1}}\,\chi_{[c_1,-\alpha]}(s),\\
			\varphi_2 (s) & = (\tau + \tau_2) 2\alpha \beta e^{2\beta\xi_2} \,  \frac{(s-\alpha)^{\beta-1}}{(s+\alpha)^{\beta+1}}\,\chi_{[\alpha,\,c_2]}(s),
			\end{align*}
			and	$\Vert \nabla r(\Bx) \Vert_{L^\infty(\RR^2)} < C$ for some constant $C$ independent of $\epsilon$.

			\item[\rm(b)] For $0<k_1, k_2 < 1$,	we have	
			\begin{align*}
			u(\Bx) = H(\Bx) + \frac{C_y r_*^2}{2} \left[ \int_{[\Bc_1, \Bp_1]} \frac{ \langle \Bx-\Bs, \hat{\bold{e}}_y \rangle}{|\Bx-\Bs|^2} \psi_1 (\Bs)\, d\sigma(\Bs) + \int_{[\bold{p}_2,\Bc_2]} \frac{ \langle \Bx-\Bs, \hat{\bold{e}}_y \rangle}{|\Bx-\Bs|^2} \psi_2 (\Bs)\, d\sigma(\Bs) \right]+ r(\Bx) , 
			\end{align*}
			where the charge distributions are
			\begin{align*}
			\psi_1 (s) & = (\tau + \tau_1) e^{2\beta\xi_1}  \left(\frac{s+\alpha}{s-\alpha}\right)^\beta \chi_{[c_1,-\alpha]}(s),\\
			\psi_2 (s) & = (\tau + \tau_2) e^{2\beta\xi_2}  
			\left(\frac{s-\alpha}{s+\alpha}\right)^\beta \chi_{[\alpha,\,c_2]}(s),
			\end{align*}
			and
			$\Vert \nabla r(\Bx) \Vert_{L^\infty(\RR^2)} < C$ for some constant $C$ independent of $\epsilon$.
		\end{itemize}
	\end{theorem}
	
	One can derive a similar line charge formula for $u$ in $B_1\cup B_2$ by use of the line charge formula of the Lerch transcendent function in section \ref{sec:lerch}.

	
	\begin{figure}[!]
		\centering
		\includegraphics[width=7cm,height=4cm,trim={2cm 3cm 2cm 2.5cm}, clip]{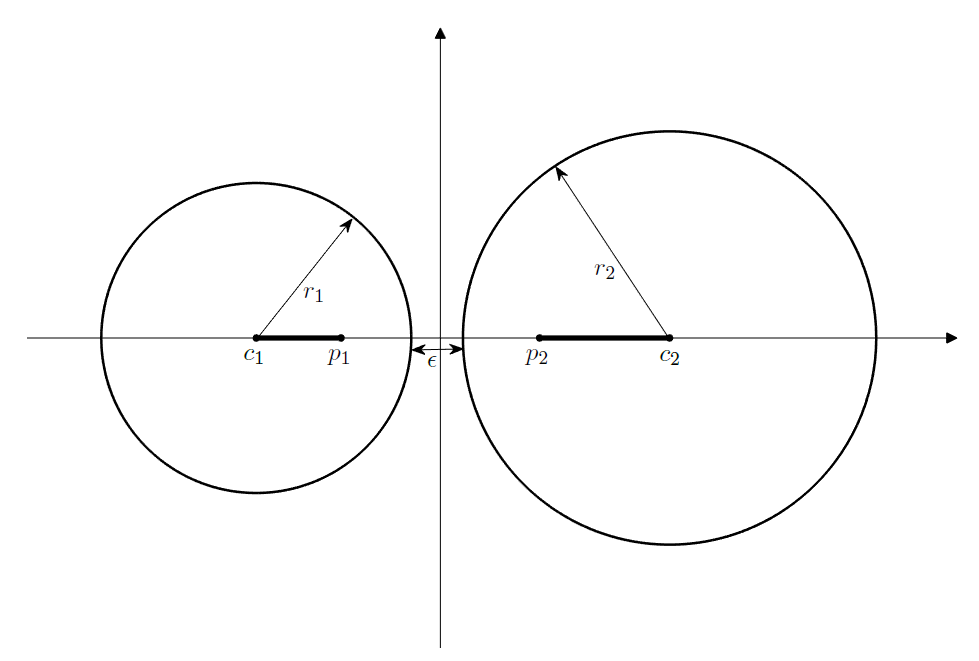}
		\caption{Geometry of two separated disks. Bold intervals indicate the support of line charges.}\label{fig:support:separated}
		\begin{subfigure}{1\textwidth}
			\centering
			\adjincludegraphics[width=5cm, height=3.75cm, trim={.8cm .5cm .5cm .4cm}, clip]{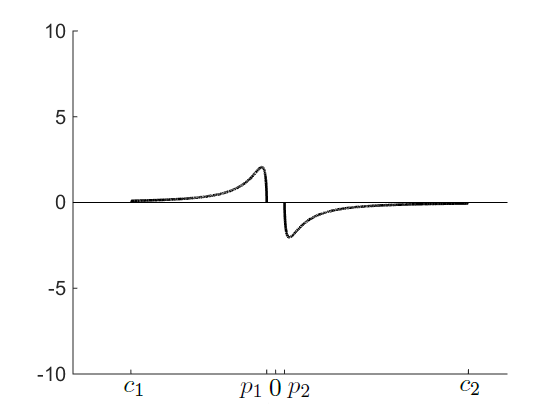}
			\hskip .5cm
			\adjincludegraphics[width=5cm, height=3.75cm, trim={.8cm .5cm .5cm .4cm}, clip]{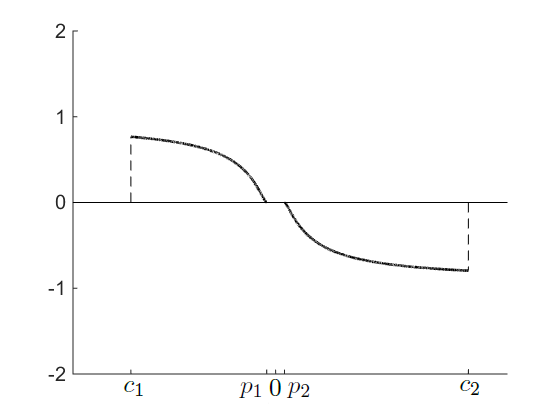}
			\caption{$k_1 = 10,\, k_2 = 15$ (left) and $k_1=1/10,\, k_2 = 1/15$ (right). $ \ep = 0.005$ ($\beta = 1.5471, \tau = 0.7159$)}
		\end{subfigure}	
		\begin{subfigure}{1\textwidth}
			\centering
			\adjincludegraphics[width=5cm, height=3.75cm, trim={.8cm .5cm .5cm .4cm}, clip]{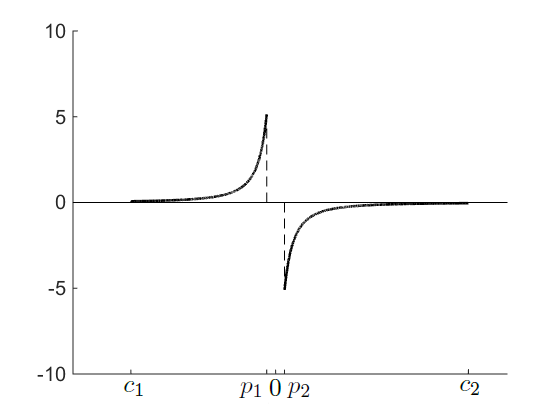}
			\hskip .5cm
			\adjincludegraphics[width=5cm, height=3.75cm, trim={.8cm .5cm .5cm .4cm}, clip]{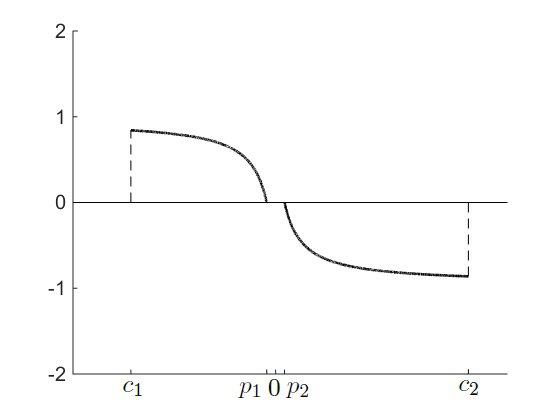}
			\caption{$k_1= 15.4471,\, k_2 = 23.1707$ (left) and $k_1=1/ 15.4471,\, k_2 =1/ 23.1707$ (right). $\ep = 0.005$ ($\beta=1, \tau = 0.8057$)}
		\end{subfigure}	
		\begin{subfigure}{1\textwidth}
			\centering
			\adjincludegraphics[width=5cm, height=3.75cm, trim={.8cm .5cm .5cm .4cm}, clip]{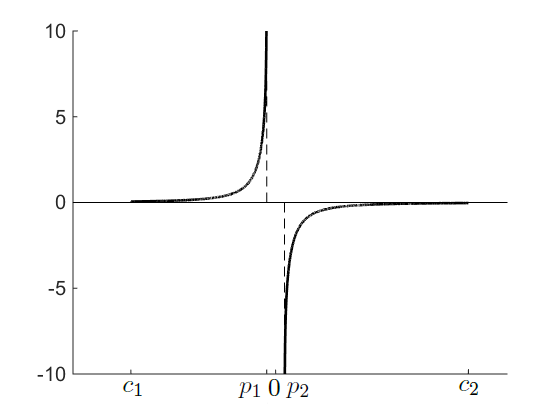}
			\hskip .5cm
			\adjincludegraphics[width=5cm, height=3.75cm, trim={.8cm .5cm .5cm .4cm}, clip]{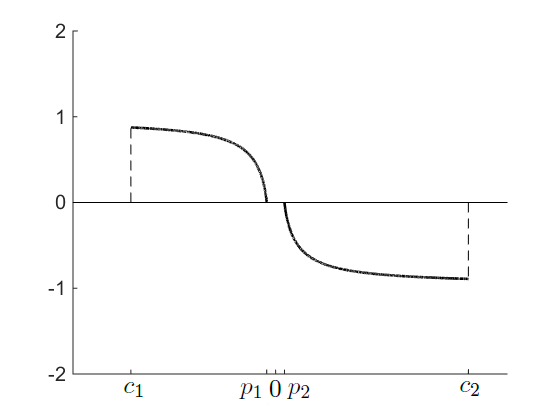}
			\caption{$k_1=20,\, k_2=30$ (left) and $k_1=1/20,\, k_2=1/30$ (right). $\ep = 0.005$ ($\beta = 0.7720, \tau = 0.8464 $)}
		\end{subfigure}	
		\caption{Image line charges for the two separated disks with $r_1=1.5,\, r_2 =2,\, \ep=0.005$ and various $k_1,k_2$-values. 
			The left column shows $\frac{1}{2}(\varphi_1-\varphi_2)$, and the right column shows $\frac{1}{2}(\psi_1-\psi_2)$ in Theorem \ref{thm:imagecharge:separate}. 
		}\label{fig:separated}
	\end{figure}

	\begin{figure}[!]
		\begin{center}
			\begin{subfigure}{1\textwidth}
				\centering
				\adjincludegraphics[width=5cm, height=3.75cm, trim={.8cm .5cm .5cm .4cm}, clip]{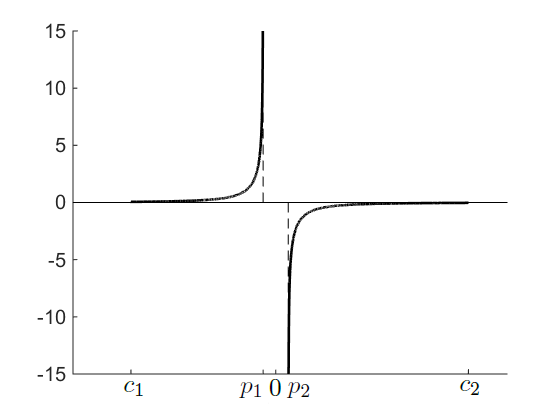}
				\hskip .5cm
				\adjincludegraphics[width=5cm, height=3.75cm, trim={.8cm .5cm .5cm .4cm}, clip]{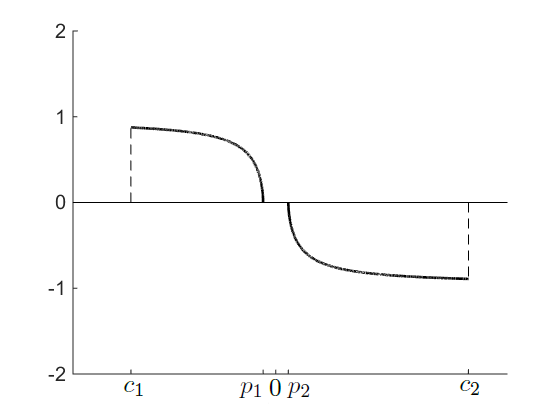}
				\caption{$\ep=0.01$ ($\beta = 0.5459, \tau = 0.8464$)}
			\end{subfigure} 
			\begin{subfigure}{1\textwidth}
				\centering
				\adjincludegraphics[width=5cm, height=3.75cm, trim={.8cm .5cm .5cm .4cm}, clip]{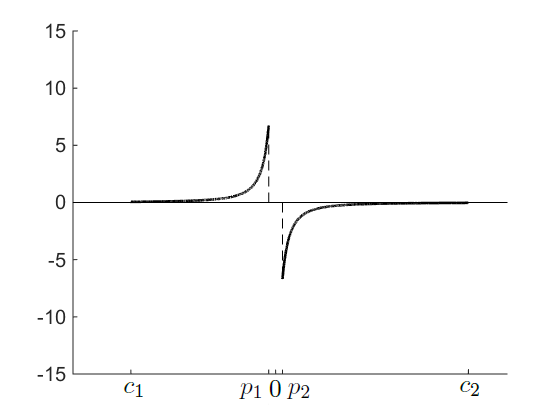}
				\hskip .5cm
				\adjincludegraphics[width=5cm, height=3.75cm, trim={.8cm .5cm .5cm .4cm}, clip]{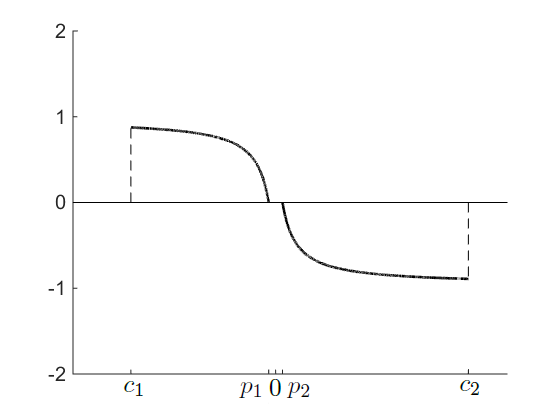}
				\caption{$\ep= 0.0030$ ($\beta=1, \tau = 0.8464$)}
			\end{subfigure} 

			\begin{subfigure}{1\textwidth}
				\centering
				\adjincludegraphics[width=5cm, height=3.75cm, trim={.8cm .5cm .5cm .4cm}, clip]{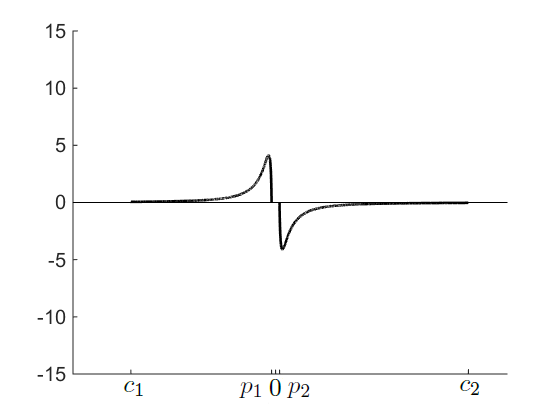}
				\hskip .5cm
				\adjincludegraphics[width=5cm, height=3.75cm, trim={.8cm .5cm .5cm .4cm}, clip]{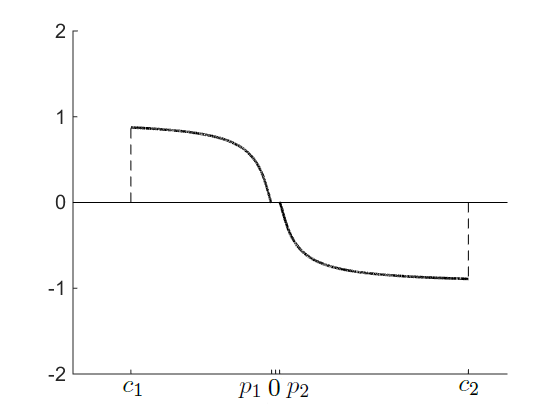}
				\caption{$ \ep=0.001$ ($\beta = 1.7263, \tau = 0.8464$)}
			\end{subfigure} 

			\caption{Image line charges for the two separated disks with $r_1=1.5,\, r_2 =2$ and various $\ep$-values, where $k_1 = 20,\, k_2 = 30$ (left column) and $k_1 = 1/20,\, k_2 = 1/30$ (right column). The left column shows $\frac{1}{2}(\varphi_1-\varphi_2)$, and the right column shows $\frac{1}{2}(\psi_1-\psi_2)$ in Theorem \ref{thm:imagecharge:separate}. As $\ep$ decreases, $\beta$ increases so that $\frac{1}{2}(\varphi_1-\varphi_2)$ becomes bounded.}
		\end{center}
	\end{figure}

	
	\section{Conclusion}
	We have analyzed the gradient blow-up of the solution to the conductivity problem in the presence of an inclusion with core-shell geometry. We showed that the gradient blow-up happens only when the conductivities of the core and shell satisfy the conditions $k_e<1$ and $k_e<k_i$. We then derived an asymptotic formula for the potential function in terms of the bipolar coordinates. We finally expressed the asymptotic formula in terms of the single and double layer potentials with line charges. It will be of interest to generalize the result to core-shell geometry with arbitrary shapes, or to the elasticity problem.



\begin{thebibliography}{10}
	
	\bibitem{ammari2013spectral}
	Habib Ammari, Giulio Ciraolo, Hyeonbae Kang, Hyundae Lee, and Graeme~W Milton.
	\newblock Spectral theory of a Neumann--Poincar{\'e}-type operator and analysis
	of cloaking due to anomalous localized resonance.
	\newblock {\em Archive for Rational Mechanics and Analysis}, 208(2):667--692,
	2013.
	
	
	\bibitem{ACKLY}
	Habib Ammari, Giulio Ciraolo, Hyeonbae Kang, Hyundae Lee, and Kihyun Yun.
	\newblock Spectral analysis of the Neumann--Poincar{\'e} operator and
	characterization of the stress concentration in anti-plane elasticity.
	\newblock {\em Archive for Rational Mechanics and Analysis}, 208(1):275--304,
	2013.
	
	\bibitem{ammari2007estimates}
	Habib Ammari, George Dassios, Hyeonbae Kang, and Mikyoung Lim.
	\newblock Estimates for the electric field in the presence of adjacent
	perfectly conducting spheres.
	\newblock {\em Quarterly of applied mathematics}, 65(2):339--356, 2007.
	
	
 \bibitem{book2}
  Habib Ammari, Josselin Garnier, Wenjia Jing, Hyeonbae Kang, Mikyoung Lim, Knut S{\o}lna, and Han Wang.
 \newblock {\em Mathematical and statistical methods for multistatic imaging (Vol. 2098)}
 \newblock Springer, 2013. 

	\bibitem{book}
	Habib Ammari and Hyeonbae Kang.
	\newblock {\em Reconstruction of small inhomogeneities from boundary
		measurements}.
	\newblock Springer, 2004.
	
	
	\bibitem{ammari2007polarization}
Habib Ammari and Hyeonbae Kang.
\newblock {\em Polarization and Moment Tensors: With Applications to Inverse Problems and Effective Medium Theory}, volume 162.
\newblock Springer Science \& Business Media, 2007.

	\bibitem{AKLLL}
	Habib Ammari, Hyeonbae Kang, Hyundae Lee, Jungwook Lee, and Mikyoung Lim.
	\newblock Optimal estimates for the electric field in two dimensions.
	\newblock {\em Journal de math{\'e}matiques pures et appliqu{\'e}es},
	88(4):307--324, 2007.
	
	\bibitem{AKLLZ}
	Habib Ammari, Hyeonbae Kang, Hyundae Lee, Mikyoung Lim, and Habib Zribi.
	\newblock Decomposition theorems and fine estimates for electrical fields in
	the presence of closely located circular inclusions.
	\newblock {\em Journal of Differential Equations}, 247(11):2897--2912, 2009.
	
	\bibitem{AKL}
	Habib Ammari, Hyeonbae Kang, and Mikyoung Lim.
	\newblock Gradient estimates for solutions to the conductivity problem.
	\newblock {\em Mathematische Annalen}, 332(2):277--286, 2005.
	
	
%
%
	

\bibitem{bab} I. Babu\u{s}ka, B. Andersson, P. Smith and K. Levin, Damage analysis of fiber composites. I. Statistical analysis on fiber scale, {\it Comput. Methods Appl. Mech. Engrg.} 172 (1999), 27--77.

\bibitem{Bat} {G. K. Batchelor and R. W. O'Brien}, {Thermal or Electrical Conduction Through a Granular Material},  {\it Proc. Roy. Soc. A} 355 (1977), 313--333



\bibitem{BC}{B. Budiansky and G. F. Carrier}, { High shear stresses in stiff fiber composites}, {\it Jour. Appl. Mech.} 51 (1984),  733--735.

	\bibitem{BLY}
	Ellen~Shiting Bao, Yan~Yan Li, and Biao Yin.
	\newblock Gradient estimates for the perfect conductivity problem.
	\newblock {\em Archive for rational mechanics and analysis}, 193(1):195--226,
	2009.
	
	\bibitem{BLY2}
	Ellen~Shiting Bao, Yan~Yan Li, and Biao Yin.
	\newblock Gradient estimates for the perfect and insulated conductivity
	problems with multiple inclusions.
	\newblock {\em Communications in Partial Differential Equations},
	35(11):1982--2006, 2010.
	
	\bibitem{bao2017optimal}
	JiGuang Bao, Hongjie Ju, and Haigang Li.
	\newblock Optimal boundary gradient estimates for Lam{\'e} systems with
	partially infinite coefficients.
	\newblock {\em Advances in Mathematics}, 314:583--629, 2017.
	
	\bibitem{bao2015gradient}
	JiGuang Bao, HaiGang Li, and YanYan Li.
	\newblock Gradient estimates for solutions of the Lam{\'e} system with
	partially infinite coefficients.
	\newblock {\em Archive for Rational Mechanics and Analysis}, 215(1):307--351,
	2015.
	
	\bibitem{bao2017gradient}
	JiGuang Bao, HaiGang Li, and YanYan Li.
	\newblock Gradient estimates for solutions of the Lam{\'e} system with
	partially infinite coefficients in dimensions greater than two.
	\newblock {\em Advances in Mathematics}, 305:298--338, 2017.
	
	\bibitem{BT}
	Eric Bonnetier and Faouzi Triki.
	\newblock Pointwise bounds on the gradient and the spectrum of the
	neumann--poincar{\'e} operator: the case of 2 discs.
	\newblock {\em Contemp. Math}, 577:81--92, 2012.
	
	
	\bibitem{Dav}
	Milford~Harold Davis.
	\newblock Two charged spherical conductors in a uniform electric field: Forces
	and field strength.
	\newblock {\em The Quarterly Journal of Mechanics and Applied Mathematics},
	17(4):499--511, 1964.
	
	\bibitem{Jef}
	GB~Jeffery.
	\newblock On a form of the solution of Laplace's equation suitable for problems
	relating to two spheres.
	\newblock {\em Proceedings of the Royal Society of London. Series A, Containing
		Papers of a Mathematical and Physical Character}, 87(593):109--120, 1912.
	
	\bibitem{KLeeY}
	Hyeonbae Kang, Hyundae Lee, and KiHyun Yun.
	\newblock Optimal estimates and asymptotics for the stress concentration
	between closely located stiff inclusions.
	\newblock {\em Mathematische Annalen}, 363(3-4):1281--1306, 2015.
	
	\bibitem{KLY}
	Hyeonbae Kang, Mikyoung Lim, and KiHyun Yun.
	\newblock Asymptotics and computation of the solution to the conductivity
	equation in the presence of adjacent inclusions with extreme conductivities.
	\newblock {\em Journal de Math{\'e}matiques Pures et Appliqu{\'e}es},
	99(2):234--249, 2013.
	
	\bibitem{KLY2}
	Hyeonbae Kang, Mikyoung Lim, and KiHyun Yun.
	\newblock Characterization of the electric field concentration between two
	adjacent spherical perfect conductors.
	\newblock {\em SIAM Journal on Applied Mathematics}, 74(1):125--146, 2014.
	
	\bibitem{kang2017quantitative}
	Hyeonbae Kang and Sanghyeon Yu.
	\newblock Quantitative characterization of stress concentration in the presence
	of closely spaced hard inclusions in two-dimensional linear elasticity.
	\newblock {\em arXiv preprint arXiv:1707.02207}, 2017.
	
	\bibitem{kang2017optimal}
	Hyeonbae Kang and KiHyun Yun.
	\newblock Optimal estimates of the field enhancement in presence of a bow-tie
	structure of perfectly conducting inclusions in two dimensions.
	\newblock {\em arXiv preprint arXiv:1707.00098}, 2017.
	
	\bibitem{keller}
	Joseph~B Keller.
	\newblock Conductivity of a medium containing a dense array of perfectly
	conducting spheres or cylinders or nonconducting cylinders.
	\newblock {\em Journal of Applied Physics}, 34(4):991--993, 1963.
	
	\bibitem{kellog}
	Oliver~Dimon Kellogg.
	\newblock {\em Foundations of potential theory}, volume~31.
	\newblock Springer Science \& Business Media, 2012.
	
	
%

\bibitem{lassiter2008close}
J. B. Lassiter, J. Aizpurua, L. I. Hernandez, D. W. Brandl, I. Romero, S. Lal, J. H. Hafner, P. Nordlander and N. J. Halas.
\newblock Close encounters between two nanoshells.
\newblock{\em Nano letters}, 8(4):1212--1218, 2008.

	
	
	
	
	
	\bibitem{lekner11}
	John Lekner.
	\newblock Near approach of two conducting spheres: enhancement of external
	electric field.
	\newblock {\em Journal of Electrostatics}, 69(6):559--563, 2011.
	
	\bibitem{li2017optimal}
	Haigang Li and Longjuan Xu.
	\newblock Optimal estimates for the perfect conductivity problem with
	inclusions close to the boundary.
	\newblock {\em arXiv preprint arXiv:1705.04459}, 2017.
	
	\bibitem{LV}
	Yan~Yan Li and Michael Vogelius.
	\newblock Gradient estimates for solutions to divergence form elliptic
	equations with discontinuous coefficients.
	\newblock {\em Archive for Rational Mechanics and Analysis}, 153(2):91--151,
	2000.
	
	\bibitem{LN}
	Yanyan Li and Louis Nirenberg.
	\newblock Estimates for elliptic systems from composite material.
	\newblock {\em Communications on pure and applied mathematics}, 56(7):892--925,
	2003.
	
	\bibitem{LYu3D}
	Mikyoung Lim and Sanghyeon Yu.
	\newblock Asymptotic analysis for superfocusing of the electric field in
	between two nearly touching metallic spheres.
	\newblock {\em arXiv preprint arXiv:1412.2464}, 2014.
	
	\bibitem{LYu2D}
	Mikyoung Lim and Sanghyeon Yu.
	\newblock Asymptotics of the solution to the conductivity equation in the
	presence of adjacent circular inclusions with finite conductivities.
	\newblock {\em Journal of Mathematical Analysis and Applications},
	421(1):131--156, 2015.
	
		\bibitem{lim2017stress}
	Mikyoung Lim and Sanghyeon Yu.
	\newblock Stress concentration for two nearly touching circular holes.
	\newblock {\em arXiv preprint arXiv:1705.10400}, 2017.
	
	\bibitem{LY}
	Mikyoung Lim and Kihyun Yun.
	\newblock Blow-up of electric fields between closely spaced spherical perfect
	conductors.
	\newblock {\em Communications in Partial Differential Equations},
	34(10):1287--1315, 2009.
	
	\bibitem{Max} {J. C. Maxwell}, {\it A Treatise on Electricity and Magnetism}, Vol. I (3rd Edn). Oxford University Press (1891), reprinted by Dover, New York (1954)


	\bibitem{McP2}
	RC~McPhedran and AB~Movchan.
	\newblock The rayleigh multipole method for linear elasticity.
	\newblock {\em Journal of the Mechanics and Physics of Solids}, 42(5):711--727,
	1994.
	
	\bibitem{McP}
	RC~McPhedran, L~Poladian, and GW~Milton.
	\newblock Asymptotic studies of closely spaced, highly conducting cylinders.
	\newblock In {\em Proceedings of the Royal Society of London A: Mathematical,
		Physical and Engineering Sciences}, volume 415, pages 185--196. The Royal
	Society, 1988.
	
	\bibitem{Pol}
	L~Poladian.
	\newblock General theory of electrical images in sphere pairs.
	\newblock {\em The Quarterly Journal of Mechanics and Applied Mathematics},
	41(3):395--417, 1988.
	
	\bibitem{Pol2}
	L~Poladian.
	\newblock Asymptotic behaviour of the effective dielectric constants of
	composite materials.
	\newblock In {\em Proceedings of the Royal Society of London A: Mathematical,
		Physical and Engineering Sciences}, volume 426, pages 343--359. The Royal
	Society, 1989.
	
	
\bibitem{Rom} {I. Romero, J. Aizpurua, G. W. Bryant, and F. Javier Garc\'ia de Abajo}, {Plasmons in nearly touching metallic nanoparticles: singular response in the limit of touching dimers},
 {\it Opt. Express} 14 (2006), 9988--9999.

	
	\bibitem{YL17}
Sanghyeon Yu and Mikyoung Lim.
\newblock Shielding at a distance due to anomalous resonance.
\newblock {\em New Journal of Physics}, 19(3):033018, 2017.
	
	
	\bibitem{AmmariYu}
	Sanghyeon Yu and Habib Ammari.
	Plasmonic interaction between nanospheres. SIAM Review, 60(2),
	356-385 (2018)
	
		
	\bibitem{Smy}
	William~B Smythe.
	\newblock Static and dynamic electricity.
	\newblock 1988.
	
	\bibitem{Y}
	Kihyun Yun.
	\newblock Estimates for electric fields blown up between closely adjacent
	conductors with arbitrary shape.
	\newblock {\em SIAM Journal on Applied Mathematics}, 67(3):714--730, 2007.
	
	\bibitem{Y2}
	KiHyun Yun.
	\newblock Optimal bound on high stresses occurring between stiff fibers with
	arbitrary shaped cross-sections.
	\newblock {\em Journal of Mathematical Analysis and Applications},
	350(1):306--312, 2009.
	
	\bibitem{Y3}{K. Yun}, An optimal estimate for electric fields on the shortest line segment between two spherical insulators in three dimensions. 	arXiv:1504.07679.


 
%
%

\end{thebibliography}
\end{document}